\documentclass[11pt]{amsart}
\usepackage{a4wide}
\usepackage{kotex}
\usepackage{amsmath}
\usepackage{amsfonts}
\usepackage{amscd}

\usepackage{amssymb}
\usepackage{graphicx}
\usepackage{dsfont}
\usepackage{color}
    \usepackage{bbm}
\usepackage[colorlinks]{hyperref}
\hypersetup{citecolor=blue,}
\hypersetup{urlcolor=black,}
\usepackage{url}

\usepackage{breakurl}

\usepackage{sseq}
\usepackage{tikz-cd}
\usepackage[abs]{overpic} 
\usepackage{bm}
\usepackage{caption}
\usepackage{subcaption}

\theoremstyle{theorem}

\newtheorem{theorem}{Theorem}[section]
\newtheorem{lemma}[theorem]{Lemma}

\newtheorem{proposition}[theorem]{Proposition}
\newtheorem{corollary}[theorem]{Corollary}

\newtheorem*{question*}{Question}

\theoremstyle{definition}
\newtheorem*{definition*}{Definition}
\newtheorem{definition}[theorem]{Definition}

\newtheorem*{example*}{Example}
\newtheorem{example}[theorem]{Example}
\newtheorem*{observation*}{Observation}

\newtheorem*{Goal*}{Goal}

\newtheorem*{Assumption*}{Assumption}

\theoremstyle{remark}
\newtheorem*{remark*}{Remark}
\newtheorem{remark}[theorem]{Remark}

\numberwithin{equation}{section}

\newcommand{\lda}{\lambda}
\newcommand{\Lda}{\Lambda}
\newcommand{\p}{\partial}

\newcommand{\C}{{\mathbb{C}}}
\newcommand{\R}{{\mathbb{R}}}

\newcommand{\Z}{{\mathbb{Z}}}
\newcommand{\N}{{\mathbb{N}}}

\DeclareMathOperator{\id}{id}

\DeclareMathOperator{\rk}{rk}

\DeclareMathOperator{\lcm}{lcm}

\DeclareMathOperator{\Ham}{Ham}

\DeclareMathOperator{\LFH}{LFH}
\DeclareMathOperator{\Symp}{Symp}

\DeclareMathOperator{\Fix}{Fix}

\DeclareMathOperator{\HW}{HW}

\DeclareMathOperator{\HF}{HF}
\DeclareMathOperator{\CF}{CF}

\DeclareMathOperator{\Ho}{H}
\DeclareMathOperator{\st}{st}

\DeclareMathOperator{\anti}{anti}
\DeclareMathOperator{\inv}{inv}
\DeclareMathOperator{\vol}{vol}

\begin{document}

\title[Volume growth in Milnor fibers]{Volume growth via real Lagrangians in Milnor fibers of Brieskorn polynomials}

\author{Joontae Kim}

\address{Department of Mathematics, Sogang University, 35 Baekbeom-ro, Mapo-gu, Seoul 04107, Republic of Korea}
\email{joontae@sogang.ac.kr}

\author{Myeonggi Kwon}

\address{Department of Mathematics Education, and Institute of Pure and Applied Mathematics, Jeonbuk National University, Jeonju 54896, Republic of Korea}
\email{mkwon@jbnu.ac.kr}

\begin{abstract}
    In this paper we study the volume growth in the component of fibered twists in Milnor fibers of Brieskorn polynomials. We obtain a uniform lower bound of the volume growth for a class of Brieskorn polynomials using a Smith inequality for involutions in wrapped Floer homology. To this end, we investigate a family of real Lagrangians in those Milnor fibers whose topology can be systematically described in terms of the join construction.
\end{abstract}

\maketitle


\section{Introduction}

A complex polynomial $f\colon \C^{n+1} \rightarrow \C$ of the form
$$
f({\mathbf z}) = z_0^{a_0}+ z_1^{a_1} +\cdots + z_n^{a_n},
$$
with $a_j \in \N$, is called a \emph{Brieskorn polynomial}. Its level set defines an exact symplectic manifold 
$$
V(\mathbf a) = V(a_0, a_1, \dots, a_n) = f^{-1}(1),
$$
called a \emph{Milnor fiber}, with respect to the standard symplectic form on $\C^{n+1}$. Taking the intersection with a  ball $B^{2n+2}$ in $\C^{n+1}$, the contact boundary
admits a canonical contact form whose Reeb flow is periodic.
As introduced in Seidel \cite[Section~4c]{Se00}, the periodic Reeb flow gives rise to a compactly supported symplectomorphism $\vartheta \in \Symp^c(V(\mathbf a))$, called a \emph{fibered twist}. For example, the fibered twist $\vartheta$ on the Milnor fiber $V(2, 2, \dots, 2) \cong T^*S^n$ of the quadratic polynomial is symplectically isotopic to the square of the generalized Dehn twist on the cotangent bundle $T^*S^n$ along the zero section. 

It is a driving question in symplectic topology to ask how much it differs from other fields, especially from smooth topology. The connected component $[\vartheta] \in \pi_0(\Symp^c(V(\mathbf a)))$ has been of particular interest as it has provided a fruitful source of examples of symplectomorphisms which are smoothly isotopic to the identity but not symplectically so. For example, it is proved in Seidel \cite[Section 4c]{Se00} that if the sum of the weights $\sum_{j} 1/a_j$ is not equal to one, then the component of fibered twists on $V(\mathbf a)$ has infinite order in $\pi_0  (\Symp^c(V(\mathbf a)))$. We refer the reader to \cite{AcuAvd16, ChiDinvan14, ChiDinvan16, KKL18, Oba22, Se00, Ul17} for further related results.

In this paper, we are interested in a quantitative way to detect the non-triviality of the component $[\vartheta] \in \Symp^c(V(\mathbf a))$ using a notion of entropy called \emph{volume growth}. In general, for a compactly supported diffeomorphism $\varphi$ on a Riemannian manifold $M$, the \emph{$i$-dimensional slow volume growth}  $s_i(\varphi)$ of $\varphi$ is defined as
$$
s_i(\varphi) = \sup_{\Delta \in \Sigma_i} \liminf_{m \rightarrow \infty} \frac{\log \vol (\varphi^m(\Delta))}{\log m} \in [0, \infty],
$$
where $\Sigma_i$ denotes the set of embedded $i$-dimensional cubes in $M$. This measures the complexity of the diffeomorphism on a polynomial scale by looking at how quickly the volume $\vol (\varphi^m(\Delta))$ of the image of an $i$-dimensional cube $\Delta$ increases under iterations of $\varphi$. 
An intriguing study of the slow volume growth in symplectic topology was made by Polterovich \cite{Pol02} who found a uniform lower bound of $s_1(\varphi)$ for every symplectomorphism $\varphi\neq \id$ on a closed symplectic manifold with vanishing second homotopy group.




The main result of this paper is to obtain a uniform lower bound of the half dimensional slow volume growth $s_n(\varphi)$ in the component of fibered twists $\vartheta$ on a family of Milnor fibers of Brieskorn polynomials:

\begin{theorem}\label{thm: theorem A}
Let $V(\mathbf{a}) = V(a_0, a_1, \dots, a_n)$ be a Milnor fiber of a Brieskorn polynomial with $n \geq 3$ such that $a_j = 2$ for at least three $a_j$'s 
and at most one $a_j$ is odd. Let $\vartheta$ be a fibered twist on $V(\mathbf{a})$. Then $s_n(\varphi) \geq 1$ for any compactly supported symplectomorphism $\varphi$ such that $[\varphi] = [\vartheta^k] \in \pi_0(\Symp^c(V(\mathbf{a})))$ for some $k \neq 0$.
\end{theorem}

This generalizes the result in \cite[Section~7]{KKL18} for fibered twists on Milnor fibers of $A_k$-type polynomials i.e. the case when $a_0 = k+1$ with $k \geq 1$ and $a_j = 2$ for $1 \leq j \leq n$.

The uniform lower bound is particularly interesting when (a power of) a fibered twist $\vartheta$ acts trivially on homology. 
In this case one cannot find such a uniform lower bound from topological information e.g. in the sense of Yomdin \cite[Theorem~1.1]{Yom87}. Therefore it can be seen as a symplectic phenomenon. For example, a fibered twist on the Milnor fiber $V(3, 4, 2, 2, 2, 2)$, which satisfies the assumptions in Theorem~\ref{thm: theorem A}, indeed acts trivially on homology, see Remark~\ref{rem: 34222}. In addition, for $A_{k}$-type polynomials with $n$ even, a power of a fibered twist $\vartheta$ is smoothly isotopic to the identity, as noted in \cite[Remark~7.2]{KKL18}. The proof of Theorem~\ref{thm: theorem A} also implies that the component of fibered twists has infinite order in $\pi_0  (\Symp^c(V(\mathbf a)))$, see Remark \ref{rem: infinite order}.




The geometric idea of the result goes back to a seminal work of Frauenfelder--Schlenk \cite{FraSch05}. They obtained a uniform lower bound of the slow volume growth $s_n(\varphi)$ in the component of fibered twists on a certain class of cotangent bundles  by studying intersections of Lagrangian submanifolds. More precisely, $s_n(\varphi)$ is bounded from below by the volume growth of a cotangent fiber $L$ under iteration of $\varphi$, and this Lagrangian volume growth can be detected by an algebraic growth rate of the dimension of the Lagrangian (intersection) Floer homology of $L$ and $\varphi^m(L)$. Invariance properties of Lagrangian Floer homology play an essential role here to obtain a uniform lower bound. 

To prove Theorem~\ref{thm: theorem A}, we generalize this idea to Liouville domains $W$, with $\Ho_c^1(W; \R) = 0$,  whose Reeb flow on the boundary is periodic: 


\begin{theorem}\label{thm: intro_volumegrowthviagrowthrate}
If there exists an admissible Lagrangian $L \subset W$ whose linear growth rate $\Lambda(L)$ is positive, then for any compactly supported symplectomorphism $\varphi$ with $\varphi \in [\vartheta^k] \in \pi_0(\Symp^c(\widehat W))$ for some $k \neq 0$, we have $s_n(\varphi) \geq  1$.
\end{theorem}

Here, the linear growth rate $\Lambda(L)$ of an admissible Lagrangian $L$ measures how quickly the dimension of the filtered wrapped Floer homology $\HW^a(L)$ increases along the action filtration; see Section~\ref{sec: growthrate} for a precise definition. 

It is worth pointing out that Theorem~\ref{thm: intro_volumegrowthviagrowthrate} is proved in \cite[Theorem~A]{KKL18} when the admissible Lagrangian $L$ is diffeomorphic to the $n$-ball. Manipulating the Crofton inequality for Lagrangian tomographs, which was recently introduced in \cite{CiGiGu21}, we give a proof of Theorem~\ref{thm: intro_volumegrowthviagrowthrate} without any assumption on the diffeomorphism type of $L$, see Section~\ref{sec: volumegrwothoffiberedtwists}.

A crucial step of the proof of Theorem~\ref{thm: theorem A} is to find admissible Lagrangians in Milnor fibers of positive growth rate. For this, we employ the canonical $\Z_{a_0} \oplus \Z_{a_1} \oplus \cdots \oplus \Z_{a_n}$-symmetry on the Milnor fiber $V(\mathbf a) = V(a_0, a_1, \dots, a_n)$ given by
\begin{equation}\label{eq: syminV}
(m_0, m_1, \dots, m_n) \cdot {\mathbf z} = (e^{{m_0 2\pi i}/a_0}z_0, e^{{m_1 2\pi i}/a_1}z_1, \dots, e^{{m_n 2\pi i}/a_n}z_n)
\end{equation}
for $0 \leq m_j \leq a_j-1$. Observe that complex conjugation
$$
\mathcal{R}(z_0, z_1, \dots, z_n) = (\overline z_0, \overline z_1, \dots, \overline z_n)
$$
defines an anti-symplectic involution on $V(\mathbf a)$.
Composing $\mathcal{R}$ with the symmetry \eqref{eq: syminV}, we can produce a family of anti-symplectic involutions $\mathcal{R}_{\mathbf m} = \mathcal{R}_{(m_0, m_1, \dots, m_n)}$ on $V(\mathbf a)$, and hence we obtain a family of real Lagrangians $L_{\mathbf{m}} = \Fix \mathcal{R}_{\mathbf m}$ in $V(\mathbf a)$.
When~$n=0$, we denote $L_{m}=\Fix \mathcal{R}_{m}$ in $V(a)$.
See Section~\ref{sec: realLag} for a more detailed description.
These real Lagrangians turn out to be particularly useful for several reasons, and one of them is that their topology can be explicitly understood as follows.
\begin{theorem}\label{thm: realjoinintro}
The real Lagrangian $L_{\mathbf{m}} = L_{(m_0, \dots, m_n)}$ is homotopy equivalent to the join $L_{m_0} * L_{m_1} * \cdots * L_{m_n}$ i.e.
$$
L_{\mathbf{m}} \simeq L_{m_0}* L_{m_1} * \cdots * L_{m_n}.
$$
In particular, $L_{\mathbf m}$ is either contractible or homotopy equivalent to the $k$-dimensional sphere $S^k$ for some $0 \leq k \leq n$ if it is nonempty.
\end{theorem} 
The real Lagrangian $L_{\mathbf m}$ can be seen as a real variety of Brieskorn type singularities, which might be of independent interest. The above theorem can be regarded as a real analogue of a classical result in singularity theory by Sebastiani--Thom \cite{SebTho71}; see also Oka \cite{Oka73}.

One of the real Lagrangians in $A_k$-type Milnor fibers $V(k+1, 2, \dots, 2)$, given by the anti-symplectic involution ${\mathbf z} \mapsto (\overline z_0, -\overline z_1, -\overline z_2,  \dots , -\overline z_n)$, is intensively studied in \cite{BaKw21, KKL18}. Under an identification of $V(k+1, 2, \dots, 2)$ with the $k$-linear plumbing of $T^*S^n$, this Lagrangian corresponds to (two copies of) a cotangent fiber \cite[Section~3.3.3]{BaKw21}. It is shown in \cite[Section~7]{KKL18} that it has positive linear growth rate via explicit computation of the filtered wrapped Floer homology. This was possible as the differentials for the chain complex vanish just for degree reasons. For general real Lagrangians, however, one can hardly expect to compute the differentials for the filtered wrapped Floer homology. 

In this paper, we instead exhibit a more general and structural approach to detect the growth rate. Namely, we manipulate a Smith inequality for involutions in wrapped Floer homology over $\Z_2$-coefficients. This is deduced from a localization principle in $\Z_2$-equivariant Lagrangian Floer homology introduced by Seidel--Smith \cite{SS}.
Observe that a Milnor fiber $V(\mathbf a ) = V(a_0, a_1, \dots, a_n)$, with $a_n$ even, admits a $\Z_2$-symmetry given by the symplectic involution 
\begin{equation}\label{eq: symp_invol}
\sigma(z_0, \dots, z_{n-1}, z_n) = (z_0, \dots, z_{n-1}, -z_n).
\end{equation}
Its fixed point set by definition can be identified with the lower dimensional Milnor fiber $V(a_0, a_1, \dots, a_{n-1})$. 
A real Lagrangian $L : =  L_{(m_0, \dots, m_n)}$ then corresponds to its invariant part, namely the lower dimensional real Lagrangian $L^{\inv} : = L_{(m_0, \dots, m_{n-1})}$. In this situation, the Smith inequality tells us that  the growth rate $\Lambda(L)$ is bounded from below by the growth rate $\Lambda(L^{\inv})$ in the fixed point set:
\begin{theorem}\label{thm: intro_locHW}
Under a stably trivial normal condition, the growth rate of $L^{\inv}$ bounds the growth rate of $L$ from below i.e. $\Lambda(L) \geq \Lambda(L^{\inv})$.
\end{theorem}
See Section~\ref{sec: growthrateofthereallagrangians} for the precise statement. This in particular implies that if the invariant Lagrangian $L^{\inv}$ has positive growth rate in wrapped Floer homology, then so has the given Lagrangian $L$ in question. Basically, in the proof of Theorem~\ref{thm: theorem A}, we take a real Lagrangian in $V(k+1, 2, 2, 2)$ as the invariant Lagrangian $L^{\inv}$ in the Smith inequality, whose linear growth rate is shown to be positive in \cite[Section~7]{KKL18} by direct computations. This is why we assume the technical assumption on the exponents that $a_j = 2$ for at least three $a_j$'s. In addition, we also assumed in Theorem~\ref{thm: theorem A} that most of the exponents are even, and this is for the $\Z_2$-symmetry $\sigma$ in \eqref{eq: symp_invol} to be well-defined on Milnor fibers inductively. It is necessary to apply the Smith inequality in Theorem~\ref{thm: intro_locHW} over $\Z_2$-coefficients. 

\begin{remark}
One can expect to establish, for example, a $\Z_3$-symmetry version of the Smith inequality, which would apply to wrapped Floer homology over $\Z_3$-coefficients, to cover a broader class of exponents in Theorem~\ref{thm: theorem A}. The real Lagrangians $L_{\mathbf m}$ however may not be invariant under the $\Z_3$-symmetry, so one would need to find another class of Lagrangians in this case.
\end{remark}



The assumption in Theorem \ref{thm: intro_locHW} comes from the notion of \emph{stably trivial normal structures} introduced in \cite[Definition~18]{SS}.
This is a sufficient condition to achieve a relevant transversality in a $\Z_2$-equivariant setup for Lagrangian Floer homology, and it is fairly nontrivial to check this assumption in practice. In Section~\ref{sec: stabilytrivialreal}, we observe that the real Lagrangians in Milnor fibers provide fruitful examples which admit stably trivial normal structures with respect to the involution $\sigma$ in \eqref{eq: symp_invol}. This largely relies on the simple nature of the topology of Milnor fibers and of the real Lagrangians. In particular, real Lagrangians which are contractible play a key role for this purpose, and we obtain such Lagrangians in view of the description of the homotopy type in Theorem~\ref{thm: realjoinintro}. Technically, we follow the idea from Hendricks \cite{Hen12} which provides a practical way to guarantee the existence of stably trivial normal structures in certain circumstances.

\subsection*{Organization of the paper} In Section~\ref{sec: GRWF}, we briefly recall basic notions in filtered wrapped Floer homology and the definition of the linear growth rate. In Section~\ref{sec: LOCAL}, we establish a Smith inequality in wrapped Floer homology, and we prove Theorem~\ref{thm: intro_locHW}. We next define a family of real Lagrangians in Milnor fibers of Brieskorn polynomials in Section~\ref{sec: realLag}, and the characterization of their topology in terms of the join construction is given in Section~\ref{sec: TPRL}. The setup to apply the Smith inequality with real Lagrangians is described in Section~\ref{sec: stabilytrivialreal}, and using this, we prove the positivity of the growth rate of some real Lagrangians in Theorem~\ref{thm: mainthmgrowthrateHW}. In Section~\ref{sec: volumegrwothoffiberedtwists}, we finally obtain a uniform lower bound for the slow volume growth in Theorem~\ref{thm: theorem A}.


\section{Growth rate via Smith inequality} \label{section: GRVL}

\subsection{Growth rate of wrapped Floer homology}  \label{sec: GRWF}

\subsubsection{Wrapped Floer homology} Let $(W, \lda)$ be a Liouville domain with a Liouville form $\lda$. This means that $(W, d\lda)$ is a symplectic manifold with boundary $\p W$, and the associated Liouville vector field $X$, defined by $\iota_{X} d\lda = \lda$, is pointing outward along $\p W$. The restricted 1-form $\alpha : = \lda|_{\p W}$ defines a contact structure $\xi : = \ker \alpha$ on $\p W$. A Lagrangian $L$ of $W$ is called \emph{admissible} if $L$ is exact (i.e. $\lda|_{L}$ is exact), $L$ intersects $\p W$ in a Legendrian of $(\p W, \xi)$, and the Liouville vector field $X$ is tangent to $TL$ near the boundary $\p L$. 

\begin{example}\label{ex: ball}
Consider the closed ball $B^{2n} : = \{\mathbf{x} + i \mathbf{y} \in \C^n \;|\; |\mathbf{x}|^2 + |\mathbf{y}|^2 \leq 1\}$ with the standard Liouville form $\lda_{\st} : = \frac{1}{2} (\mathbf{x} d \mathbf{y} - \mathbf{y} d \mathbf{x})$. The associated Liouville vector field is given by the radial field $X = \frac{1}{2}(\mathbf{x} \p_{\mathbf{x}} +  \mathbf{y} \p_{\mathbf{y}})$ which is pointing outward along the boundary $\p B^{2n} = S^{2n-1}$. The subset $L = \{\mathbf{x} + i \mathbf{y} \in B^{2n} \;|\; \mathbf{y} =0\}$ of $B^{2n}$ is an admissible Lagrangian. Its boundary $\p L$ forms the standard Legendrian sphere $S^{n-1}$ in the standard contact sphere $S^{2n-1}$. 
\end{example}

We briefly give the definition of the wrapped Floer homology $\HW(L)$ of an admissible Lagrangian $L \subset W$ mainly following \cite{Rit}; see also \cite{AboSei}. Complete the Liouville domain and the Lagrangian by
$$
\widehat W : = W \cup_{\p W} (\R_{\geq 1} \times \p W), \quad \widehat L : = L \cup_{\p L} (\R_{\geq 1} \times \p L)
$$  
where we glue them along the boundary $\p W$ via the Liouville flow. The completion $\widehat L$ is a Lagrangian of $\widehat W$ with respect to the completed Liouville form $\widehat \lda  : = \lda \cup_{\p W} (r \alpha)$ on $\widehat W$, where $r \in \R_{\geq 1}$. A Hamiltonian $H\colon \widehat W \rightarrow \R$ is called \emph{admissible} if $H$ is $C^2$-small inside $W$, $H$ is of the form $H(r) = \tau r + b$ on the cylindrical end for some $\tau, b \in \R$, where $\tau > 0$ is not a Reeb chord period on the boundary $(\p W, \p L)$. In addition, $H$ is called \emph{nondegenerate} if every Hamiltonian $1$-chord relative to $\widehat L$ is nondegenerate. Here $\tau$ is called the \emph{slope} of $H$. We denote the set of Hamiltonian 1-chords relative to $\widehat L$ by $\mathcal{P}_{\widehat L}(H)$. For a nondegenerate admissible Hamiltonian $H$, we define the \emph{Floer chain complex} $\CF(H)$ by 
$$
\CF(H) = \bigoplus_{x \in \mathcal{P}_{\widehat L}(H)}  \Z_2 \langle x \rangle.
$$
As fairly standard in Lagrangian Floer theory, the differential $\p \colon \CF(H) \rightarrow \CF(H)$ is defined by counting Floer strips connecting two Hamiltonian 1-chords with respect to an admissible family of compatible almost complex structures on $\widehat W$. We refer the reader to \cite[Section~4.4]{Rit} for more details. The homology of the chain complex $(\CF(H), \p)$ is called the \emph{wrapped Floer homology of the Hamiltonian $H$} and denoted by $\HF(H)$.

For general (possibly degenerate) admissible Hamiltonian $H$, we define $\HF(H)$ to be that of a nondegenerate generic perturbation of $H$. By the invariance property of wrapped Floer homology of Hamiltonians under Hamiltonian isotopies as in \cite[Section~4.7]{Rit}, $\HF(H)$ does not depend on the choice of nondegenerate perturbations, and hence is well-defined.
 
For two admissible Hamiltonians $H^{\tau_1}$ and $H^{\tau_2}$ of slopes $\tau_1 \leq \tau_2$, we have a canonical homomorphism $\HF(H^{\tau_1}) \rightarrow \HF(H^{\tau_2})$, called a continuation map. This yields a direct system. The \emph{wrapped Floer homology} $\HW(L)$ of the Lagrangian $L$ is defined to be the direct limit
$$
\HW(L) : = \varinjlim_{\tau \rightarrow \infty} \HF(H^{\tau}).
$$
  
\subsubsection{Action filtration}
The chain complex $\CF(H)$ admits a canonical filtration via the Hamiltonian action functional $
\mathcal{A}_H\colon \Lda_{\widehat L}\widehat W \rightarrow \R$ defined by
\begin{equation}\label{eq: actionfunctional}
\mathcal{A}_H(x) = -\int_{[0, 1]} x^* \lda - \int_0^1 H(x(t))dt + f(x(1)) - f(x(0)).
\end{equation}
Here,  $\Lda_{\widehat L}\widehat W$ is the free path space of $\widehat W$ relative to $\widehat L$, and $f\colon \widehat L \rightarrow \R$ is a primitive of $\widehat \lda|_{\widehat L}$. For $a \in \R$, the filtered chain complex $\CF^{a}(H)$ is a vector space over $\Z_2$ defined by
$$
\CF^{a}(H) = \bigoplus_{\substack{x \in \mathcal{P}(H)\\ \mathcal{A}_H(x) < a}} \Z_2 \langle x \rangle. 
$$
The filtered wrapped Floer homology $\HF^a(H)$ of $H$ is defined to be the homology of $(\CF^{a}(H), \p)$. We define the \emph{filtered wrapped Floer homology $\HW^a(L)$ of the Lagrangian $L$} by taking the direct limit
$$
\HW^a(L) = \varinjlim_{\tau \rightarrow \infty} \HF^a(H^{\tau}).
$$

\begin{remark}\label{rem: fHWcanisoHF}
For each $a \in \R$, we have a canonical isomorphism
$$
\HW^a(L) \cong \HF(H^{\tau_a})
$$
where $H^{\tau_a}$ is an admissible Hamiltonian of slope $\tau_a$ which is sufficiently close to $a$. This follows from a careful choice of a cofinal family of admissible Hamiltonians for $\HW^a(L)$.
\end{remark}

\subsubsection{Linear growth rate} \label{sec: growthrate}  The following notion measures how quickly the dimension (over $\Z_2$) of the filtered wrapped Floer homology $\HW^a(L)$ increases along the action filtration.

\begin{definition}\label{def: growthrate}
The \emph{linear growth rate} $\Lambda(L)$ of an admissible Lagrangian $L$ is defined by
$$
\Lambda(L) = \liminf_{a \rightarrow \infty} \frac{\dim \HW^a(L)}{a} \in [0, \infty].
$$
\end{definition}

We remark that $\HW^a(L)$ is a finite dimensional vector space over $\Z_2$; this follows from \cite[Lemma~4.3]{Rit} and Remark~\ref{rem: fHWcanisoHF}. In addition, if the Reeb flow on the contact boundary $(\p W, \alpha)$ is periodic, then $\Lambda(L) < \infty$; this	 means that the wrapped Floer homology has at most linear growth over the action filtration.

\begin{remark}
Following \cite[Definition~2.4]{Mc18} or \cite[Definition~4.1]{Sei08}, one can define an alternative notion of linear growth rate by
$$
\widetilde{\Lambda}(L) : = \limsup_{a \rightarrow \infty} \frac{\rk(\HW^a (L) \rightarrow \HW(L))}{a}
$$
where $\rk(\HW^a (L) \rightarrow \HW(L))$ denotes the rank of the canonical map $\HW^a (L) \rightarrow \HW(L)$ induced by the inclusion. A good property of $\widetilde{\Lambda}(L)$ is that it is an invariant up to Liouville isomorphism \cite[Lemma~4.2]{Sei08}. We nonetheless shall use $\Lambda(L)$ as it fits better for our application to volume growth in Theorem~\ref{thm: intro_volumegrowthviagrowthrate}.
\end{remark}

\subsection{Smith inequality}  \label{sec: LOCAL} In this section, we establish a Smith inequality in filtered wrapped Floer homology with an application to the growth rate of Lagrangians. 

\subsubsection{Stably trivial normal structures} \label{sec: def_stabnormal} We first briefly recall the notion of stably trivial normal structures, introduced by Seidel--Smith in \cite[Definition~18]{SS}. Let $(W, L_0, L_1, \sigma)$ be a quadruple consisting of a Liouville domain $(W, \lda)$, two admissible Lagrangians $L_0, L_1 \subset W$, and an exact symplectic involution $\sigma\colon W \rightarrow W$; that is, $\sigma$ is a diffeomorphism with $\sigma^2 = \id$ and $\sigma^* \lda = \lda$. The fixed point set $S: = \Fix \sigma$ forms a (possibly empty) symplectic submanifold of $W$ which is itself a Liouville domain with Liouville form $\lda_S : = \lda|_S$. Denote its co-dimension by $2n_{\anti}$. Below, we use the following notations.
\begin{itemize}
\item $L_i^{\inv} := L_i \cap S$.
\item With respect to the projection $[0, 1] \times S \rightarrow S$, the bundle $TW^{\anti}$ denotes the pullback of the normal bundle $N_{W}S$ of $S$ in $W$.
\item $TL_i^{\anti} := N_{L_i}L_i^{\inv}$.
\end{itemize}

\begin{definition}
A \emph{stably trivial normal structure} for $(W, L_0, L_1, \sigma)$ consists of 
\begin{itemize}
\item a unitary trivialization
$$
\Phi \colon  TW^{\anti} \oplus \C^N \rightarrow \C^{n_{\anti} + N}
$$
over $S$ for some $N \geq 0$;
\item a Lagrangian subbundle $\Lambda_i \subset (TW^{\anti} \oplus \C^N)|_{[0, 1] \times L^{\inv}_i}$ for each $i$, satisfying
\begin{align*}
&\Lambda_0|_{\{0\} \times L_0^{\inv}} = TL_0^{\anti} \oplus \R^N, \quad \;\Phi(\Lambda_0|_{\{0\} \times L_0}) = \R^{n_{\anti} + N}, \\
&\Lambda_1|_{\{1\} \times L_1^{\inv}} = TL_1^{\anti} \oplus i\R^N, \quad \Phi(\Lambda_1|_{\{1\} \times L_1}) = i\R^{n_{\anti} + N}. 
\end{align*}
\end{itemize}
\end{definition}

Under the existence of a stably trivial normal structure, the following Smith inequality from \cite[Theorem~1]{SS} holds  in Lagrangian Floer homology. 

\begin{theorem}[Seidel--Smith]\label{thm: loc_LFH}
If $(W, L_0, L_1, \sigma)$ admits a stably trivial normal structure, then 
$$
\dim \HF(L_0, L_1) \geq \dim \HF(L_0^{\inv}, L_1^{\inv}).
$$
\end{theorem}

\begin{remark}
In Theorem~\ref{thm: loc_LFH}, the Lagrangians $L_0$ and $L_1$ are not assumed to intersect transversely.
\end{remark}


\subsubsection{Smith inequality in filtered wrapped Floer homology}  Let $(W, \lda)$ be a Liouville domain with an exact symplectic involution $\sigma\colon W \rightarrow W$. Denote the fixed point set by $S = \Fix \sigma$. Let $L \subset W$ be an admissible Lagrangian, and denote $L_S : = L \cap S$, which is an admissible Lagrangian in $(S, \lda : = \lda|_{S})$. Note that the completion $(\widehat S, \widehat L_S)$ naturally embeds into the completion $(\widehat W, \widehat L)$. The involution $\sigma\colon W \rightarrow W$ extends to an exact symplectic involution on the completion $\widehat W$ by 
\begin{equation}\label{eq: sigmaextension}
\sigma(r, y) = (r, \sigma|_{\p W}(y))
\end{equation}
where $(r, y) \in \R_{\geq 1} \times \p W$.


\begin{example}
Consider the standard symplectic ball $(B^{2n}, \lda_{\st})$ and the Lagrangian $L \subset B^{2n}$ as in Example~\ref{ex: ball}. The map $\sigma\colon \C^n \rightarrow \C^n$ given by ${\mathbf z} \mapsto (z_1, \dots, z_{n-1}, - z_n)$ defines an exact symplectic involution on $(B^{2n}, \lda_{\st})$.
The fixed point set $S = \Fix \sigma \subset B^{2n}$ is given by $S = \{{\mathbf z} \in B^{2n} \;|\; z_n =0\}$.
With the restricted form $\lda_S = \lda_{\st}|_{S}$, the pair $(S, \lda_{S})$ can be identified with the closed $(2n-2)$-dimensional ball $(B^{2n-2}, \lda_{\st})$ with the standard Liouville form $\lda_{\st}$ on $B^{2n-2}$.
The intersection $L_S = L \cap S$ is the Lagrangian in $S$ given by $\{{\mathbf z} \in L \;|\; z_n = 0\}$.	
\end{example}

A Hamiltonian $H \colon \widehat W \rightarrow \R$ is called \emph{$\sigma$-invariant} if $H \circ \sigma = H$. If $H$ is a $\sigma$-invariant admissible Hamiltonian, then the restriction $H_S : = H|_{\widehat S}\colon \widehat S \rightarrow \R$ defines an admissible Hamiltonian on $\widehat S$. The associated Hamiltonian flows $\phi_{H}^t$ and $\phi_{H_S}^t$ are related by $\phi_{H_S}^t = \phi_H^t|_{\widehat S}$. 

\begin{remark}\label{rem: exofsiginvHam}
Each admissible Hamiltonian gives rise to a $\sigma$-invariant admissible Hamiltonian of the same slope. Let $H\colon \widehat W \rightarrow \R$ be an admissible Hamiltonian of slope $\tau$. Define a new Hamiltonian $H_{\sigma}\colon \widehat W \rightarrow \R$ by
$$
H_{\sigma}(z) : = \frac{1}{2} \left(H(z) + H(\sigma(z)) \right).
$$ 
Then $H_{\sigma}(\sigma(z)) = \frac{1}{2}(H(\sigma(z)) + H(\sigma^2(z))) =  \frac{1}{2}(H(\sigma(z)) + H(z)) =  H_{\sigma}(z)$ i.e. $H_{\sigma}$ is $\sigma$-invariant. 
Since $H$ is admissible, we know $H(r, y) = H(r)$ for large $r$, and it follows from \eqref{eq: sigmaextension} that 
$$
H_{\sigma}(r, y)= \frac{1}{2}(H(r, y) + H(r, \sigma|_{\p W}(y))) = \frac{1}{2}(H(r) + H(r)) = H(r).
$$
Therefore $H_{\sigma}$ is admissible of the same slope as $H$. Note that $H_{\sigma}$ is not necessarily nondegenerate even if $H$ is nondegenerate.
 \end{remark}

\begin{theorem}[Smith inequality]\label{thm: local_fHW}
Suppose that for every admissible Hamiltonian $H^{\tau}$ of slope $\tau$, the quadruple $(W, L, \phi_{H^{\tau}}(L), \sigma)$ admits a stably trivial normal structure. Then for any $a \in \R$, 
$$
\dim \HW^a(L) \geq \dim \HW^a(L_S).
$$
\end{theorem}

\begin{proof}
As in Remark~\ref{rem: fHWcanisoHF}, we can take an admissible Hamiltonian $H^{\tau}$ of slope $\tau$ with a canonical isomorphism
$$
\HW^a(L) \cong \HF(H^{\tau}).
$$
In view of Remark~\ref{rem: exofsiginvHam}, we may assume that $H^{\tau}$ is $\sigma$-invariant. As is well-known, e.g. see \cite[Remark~4.7]{Rit}, the wrapped Floer homology $\HF(H^{\tau})$ can be identified with a Lagrangian intersection Floer homology via the correspondence between Hamiltonian 1-chords and Lagrangian intersections: 
$$
\HF(H^{\tau}) \cong \HF(L, \phi_{H^{\tau}}(L)).
$$
Now, Theorem~\ref{thm: loc_LFH} tells us that 
$$
\dim \HF(L, \phi_{H^{\tau}}(L)) \geq \dim \HF(L_S, \phi_{H^{\tau}_S}(L_S))
$$
where $H^{\tau}_S$ is the restriction $H^{\tau}|_S$. Since $H^{\tau}_S$ is an admissible Hamiltonian on $\widehat S$, we obtain an isomorphism
$$
\HF(L_S, \phi_{H^{\tau}_S}(L_S)) \cong \HF(H^{\tau}_S).
$$
Again, it follows from Remark~\ref{rem: fHWcanisoHF} that 
$$
\HF(H^{\tau}_S)  \cong \HW^{a}(L_S).
$$
This completes the proof.
\end{proof}

The following is immediate from the definition of the growth rate in wrapped Floer homology; see Theorem \ref{thm: intro_locHW}.

\begin{corollary}\label{cor: locHW}
In the situation of Theorem~\ref{thm: local_fHW}, the linear growth rate of $L_S$ bounds the linear growth rate of $L$ from below i.e. $\Lambda(L) \geq \Lambda(L_S)$.
\end{corollary}

\section{Real Lagrangians in Milnor fibers of Brieskorn polynomials} \label{sec: real_Bries}


\subsection{Real Lagrangians} \label{sec: realLag} Let $f\colon \C^{n+1} \rightarrow \C$ be a Brieskorn polynomial given by
$$
f(z) = z_0^{a_0} + z_1^{a_1} + \cdots + z_n^{a_n}
$$
where $a_j \in \Z_{\geq 1}$, and consider the Milnor fiber 
$$
V(\mathbf{a}) = V(a_0, \dots, a_n) := f^{-1}(1) = \{{\mathbf z} \in \C^{n+1} \;|\; z_0^{a_0} + z_1^{a_1} + \cdots + z_n^{a_n} = 1\}.
$$
By restricting the standard Liouville form $\lda_{\st}$ of $\C^{n+1}$ to $V(\mathbf{a})$, the Milnor fiber $(V(\mathbf{a}), \lda_{\st})$ forms a completed Liouville domain; see \cite[Remark~2.8]{Kea15}. In this section, we define a certain class of real Lagrangians in $V(\mathbf{a})$.




For each integer $a \geq 1$ and $0 \leq m \leq a-1$, define a map $\mathcal{R}_{m}^a \colon \C \rightarrow \C$ by
$$
\mathcal{R}_{m}^a (z) =e^{m\pi i/a} \cdot \overline {(e^{m \pi i/a})^{-1} \cdot z} = e^{2m\pi i/a} \cdot \overline{z}.
$$
Geometrically, $\mathcal{R}_{m}^a$ is the reflection on $\C$ about the straight line (the dashed lines in Figure~\ref{fig: reflections}) passing through the origin and $e^{m\pi i/a}$.
\begin{figure}[h]
 \centering
     \begin{subfigure}[b]{0.3\textwidth}
         \centering
        \begin{tikzpicture}


\draw[gray, ->, thick] (-1.7,0)--(1.7,0);
\draw[gray, ->, thick] (0,-1.7)--(0,1.7);

\filldraw[black] (1,0) circle (2pt);
\filldraw[black] (120:1) circle (2pt);
\filldraw[black] (240:1) circle (2pt);

\draw[dashed, ultra thick] (120:1.7)--(300:1.6);

\end{tikzpicture}

         \caption{$\Fix \mathcal{R}_2^3 = \{e^{2\pi i/3}\}$}
     \end{subfigure}
     \hfill
     \begin{subfigure}[b]{0.3\textwidth}
         \centering
         \begin{tikzpicture}


\draw[gray, ->, thick] (-1.7,0)--(1.7,0);
\draw[gray, ->, thick] (0,-1.7)--(0,1.7);

\filldraw[black] (1,0) circle (2pt);
\filldraw[black] (90:1) circle (2pt);
\filldraw[black] (180:1) circle (2pt);
\filldraw[black] (270:1) circle (2pt);

\draw[dashed, ultra thick] (0:1.7)--(180:1.8);

\end{tikzpicture}

         \caption{$\Fix \mathcal{R}_0^4 = \{-1, 1\}$}
     \end{subfigure}
     \hfill
     \begin{subfigure}[b]{0.3\textwidth}
         \centering
          \begin{tikzpicture}
\draw[dashed, ultra thick] (45:1.7)--(225:1.8);

\draw[gray, ->, thick] (-1.7,0)--(1.7,0);
\draw[gray, ->, thick] (0,-1.7)--(0,1.7);

\filldraw[black] (1,0) circle (2pt);
\filldraw[black] (90:1) circle (2pt);
\filldraw[black] (180:1) circle (2pt);
\filldraw[black] (270:1) circle (2pt);

\draw[dashed, ultra thick] (45:1.7)--(225:1.7);

\end{tikzpicture}

         \caption{$\Fix \mathcal{R}_1^4 = \emptyset$}
         \label{fig:reflections_c}
     \end{subfigure}
        \caption{Reflections on $\C$}
        \label{fig: reflections}

\end{figure}
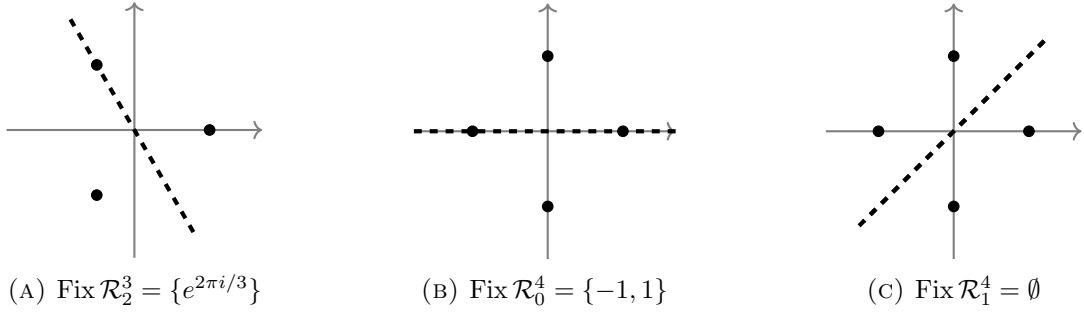
Now, for $0 \leq m_j \leq a_j-1$ and $0 \leq j \leq n$, we define the map $\mathcal{R}_{\mathbf m}^{\mathbf a} \colon \C^{n+1} \rightarrow \C^{n+1}$ coordinate-wise by
$$
\mathcal{R}_{\mathbf m}^{\mathbf a}({\mathbf z}) = \mathcal{R}^{(a_0, \dots, a_n)}_{(m_0, \dots, m_n)}({\mathbf z}) = (\mathcal{R}_{m_0}^{a_0}(z_0), \dots, \mathcal{R}_{m_n}^{a_n}(z_n)).
$$
It is straightforward to see that $\mathcal{R}_{\mathbf m}^{\mathbf a}({\mathbf z})$ is an exact anti-symplectic involution on $(\C^{n+1}, \lda_{\st})$.

\begin{lemma}\label{lem: showing_involution}
The involution $\mathcal{R}_{\mathbf m}^{\mathbf a}({\mathbf z})$ restricts to an exact anti-symplectic involution on $V(\mathbf{a})$.
\end{lemma}

\begin{proof}
It suffices to show that the restriction of $\mathcal{R}_{\mathbf m}^{\mathbf a}|_{V(\mathbf a)}$ is well-defined. By the definition of the reflection $\mathcal{R}_{m_j}^{a_j}\colon \C \rightarrow \C$, we have $
(\mathcal{R}_{m_j}^{a_j}(z_j))^{a_j} =  \overline z_j^{a_j}$. 
It follows that, for each $z \in \C^{n+1}$,
\begin{align*}
f(\mathcal{R}_{\mathbf m}^{\mathbf a}({\mathbf z})) = f(\mathcal{R}_{m_0}^{a_0}(z_0), \dots, \mathcal{R}_{m_n}^{a_n}(z_n)) &= (\mathcal{R}_{m_0}^{a_0}(z_0))^{a_0} + (\mathcal{R}_{m_1}^{a_1}(z_1))^{a_1} + \cdots + (\mathcal{R}_{m_n}^{a_n}(z_n))^{a_n} \\
&= \overline z_0^{a_0} + \cdots +  \overline z_n^{a_n} = \overline{f({\mathbf z})}.
\end{align*}
This completes the proof.
\end{proof}
\begin{remark}\label{rem: restrictstotheball}
Moreover, it is direct from the definition that $\mathcal{R}_{\mathbf{m}}^{\mathbf{a}}\colon \C^{n+1} \rightarrow \C^{n+1}$ restricts to the intersection $V(\mathbf{a}) \cap B^{2n+2}$ where $B^{2n+2} \subset \C^{n+1}$ is the closed ball. 
\end{remark}
Consequently, the fixed point set $\Fix \mathcal{R}_{\mathbf m}^{\mathbf a}$ is a  (possibly empty or disconnected) real Lagrangian in $V(\mathbf{a})$. We denote it by $L_{\mathbf{m}} = L_{\mathbf{m}}^{\mathbf{a}} : = \Fix \mathcal{R}_{\mathbf m}^{\mathbf a}$.

\begin{example}\label{ex: invol_quad}
As in \cite[Lemma~3.1]{KvK}, we can identify $V(2, \dots, 2)$ with $T^*S^n$ by the symplectomorphism
$$
{\mathbf z} = {\mathbf x}+ i{\mathbf y} \mapsto (|{\mathbf x}|^{-1}{\mathbf x}, |{\mathbf x}|{\mathbf y}).
$$
Consider the anti-symplectic involutions $\mathcal{R}_{(0, \dots, 0)}^{(2, \dots, 2)}$ and $\mathcal{R}_{(0,1 \dots, 1)}^{(2,2, \dots, 2)}$ on $V(2, \dots, 2)$ given by
$$
\mathcal{R}_{(0, \dots, 0)}^{(2, \dots, 2)} ({\mathbf z}) = (\overline z_0, \dots, \overline z_n), \quad \mathcal{R}_{(0,1 \dots, 1)}^{(2,2, \dots, 2)} ({\mathbf z})  = (\overline z_0, - \overline z_1, \dots, -\overline z_n).
$$
It is straightforward to see that the real Lagrangian $L_{(0, \dots, 0)}$ is exactly the zero section of $T^*S^n$ and the real Lagrangian $L_{(0, 1, \dots, 1)}$ is the disjoint union of two copies of fibers of $T^*S^n$ under the identification.
\end{example}

\subsection{Topology of real Lagrangians} \label{sec: TPRL} We can describe the topology of the real Lagrangian $L_{\mathbf{m}}$ in $V(\mathbf{a})$ in terms of the join of topological spaces.  

Let $I = [0, 1]$ be the closed unit interval. Recall that the \emph{join} $X_1 * \cdots * X_N$ of topological spaces $X_1, \dots, X_N$ is the quotient space given by
\begin{equation}\label{eq: joindef}
X_1 * \cdots * X_N : = \left\{(x_1, t_1, \dots, x_N, t_N) \in X_1 \times I \times \cdots \times X_N \times I \;|\; \sum_{j=1}^N t_j  = 1\right\} \bigg/ \sim
\end{equation}
where $(x_1, t_1, \dots, x_N, t_N) \sim (x_1', t_1', \dots, x_N', t_N')$ if and only if $t_j = t_j'$ for all $j$ and $x_j = x_j'$ if $t_j >0$.
\begin{remark}\label{rem: TS}
The homotopy type of $V(\mathbf{a}) = V(a_0, a_1, \dots, a_n)$ is given by the join of $0$-dimensional Milnor fibers:
\begin{equation}\label{eq: Pham}
V(\mathbf{a}) \simeq V(a_0) *V(a_1) * \cdots * V(a_n).
\end{equation}
See \cite[Corollary~1]{Oka73}, \cite{Pha65}, and \cite{SebTho71} for a more general version by Sebastiani--Thom.
\end{remark}
To describe the topology of $L_{\mathbf m}$, note that $V(a) = \{z \in \C \;|\; z^a =1\}$ consists of the $a$-th roots of unity. We denote the fixed point set of the reflection $\mathcal{R}^a_m\colon V(a) \rightarrow V(a)$ by $L_m$. The following observation is straightforward; see Figure~\ref{fig: reflections}.
\begin{proposition}\label{prop: 0dimL}
The fixed point set $L_m$ is either the empty set, a one point set, or a two points set. 
\end{proposition}

As a real analogue of \eqref{eq: Pham}, the topology of the real Lagrangian $L_{\mathbf m}$ can be written as the join of the corresponding $0$-dimensional fixed point sets $L_{m_j}$:
\begin{theorem}\label{thm:top_real}
The real Lagrangian $L_{\mathbf{m}} = L_{(m_0, \dots, m_n)}$ is homotopy equivalent to the join $L_{m_0} * L_{m_1} * \cdots * L_{m_n}$ i.e.
$$
L_{\mathbf{m}} \simeq L_{m_0}* L_{m_1} * \cdots * L_{m_n}.
$$
\end{theorem} 
Here we conventionally put $X * \emptyset = \emptyset * X : = X$ for a topological space $X$.
See also Theorem~\ref{thm: realjoinintro}. For the proof of Theorem~\ref{thm:top_real}, we will manipulate the explicit homotopy equivalences given in \cite[Proof of Theorem~1]{Oka73} for the case of the ambient spaces \eqref{eq: Pham}, focusing on its behavior with respect to the real parts. 
The following lemma is technically crucial.

\begin{lemma}\label{lem: deformtononneg}
The Lagrangian $L_{\mathbf m}$ deformation retracts onto its subset $\widetilde L_{\mathbf m}$ given by
$$
\widetilde L_{\mathbf m} : =\{{\mathbf z} \in L_{\mathbf m} \;|\; z_j^{a_j} \geq 0 \text{ for all $j$}\}.
$$
\end{lemma}

\begin{proof} Observe first that, if ${\mathbf z} = (z_0, z_1, \dots, z_n) \in L_{\mathbf{m}}$, then $z_j^{a_j} \in \R$ for all $0 \leq j \leq n$ since $z_j \in \Fix \mathcal{R}_{m_j}^{a_j}$. In light of Step 2 in \cite[Proof of Theorem~1]{Oka73}, we define a deformation retraction $G\colon L_{\mathbf{m}} \times [0, 1] \rightarrow L_{\mathbf{m}}$ from $L_{\mathbf{m}}$ onto $\widetilde L_{\mathbf m}$ as follows.
\begin{itemize}
\item For points ${\mathbf z} \in L_{\mathbf{m}}$ with $k+1$ negative $z_j^{a_j}$, assuming without loss of generality $z_{0}^{a_0}, \dots, z_k^{a_k} < 0$ for some $0 \leq k \leq n$, we define $G({\mathbf z}, t) : = (z_0(t), \dots, z_n(t))$ where
$$
z_j(t) = \begin{cases} (1-t)^{\frac{1}{a_j}}z_j & \text{for $0 \leq j \leq k$}, \\
				\displaystyle \left(\frac{1- g(z_0(t), \dots, z_k(t))}{1-g(z_0, \dots, z_k)} \right)^{\frac{1}{a_j}} z_j & \text{for $k+1 \leq j \leq n$},			
	\end{cases}
$$ 
where $g(z_0, \dots, z_k) := z_0^{a_0} + \dots + z_k^{a_k}$.
\item For ${\mathbf z}$ with $z_j^{a_j} \geq 0$ for all $j$, we define $G$ to be the identity i.e. $G({\mathbf z}, t) = {\mathbf z}$.
\end{itemize}
Note that $z_j(t)$ is a real multiple of $z_j$ for all $j$, so $\mathcal{R}_{m_j}^{a_j}(z_j(t)) = z_j(t)$, and a direct computation shows that $\displaystyle \sum_{j=0}^{n} z_j(t)^{a_j} = 1$. Therefore the map $G\colon L_{\mathbf{m}} \times [0, 1] \rightarrow L_{\mathbf{m}}$ is well-defined.
Moreover, from the definition, $G = \id$ for $t = 0$ and $G({\mathbf z}, 1) \in \widetilde L_{\mathbf{m}}$. It is straightforward to see that $G$ is continuous. It follows that $G$ is a deformation retraction of $L_{\mathbf{m}}$ onto $\widetilde L_{\mathbf{m}}$. 
\end{proof}

\begin{lemma}\label{lem: emptyforsomej}
Suppose that $L_{m_j} = \emptyset$ for some $j$. Then if ${\mathbf z} = (z_0, \dots, z_n) \in \widetilde L_{\mathbf m}$, then $z_j = 0$. 
\end{lemma}

\begin{proof}
Recall that $L_{m_j} = \{z \in \C \;|\; z^{a_j} = 1,\; \mathcal{R}_{m_j}^{a_j}(z) = z\}$.
It is direct from the definition of the reflection $\mathcal{R}_{m_j}^{a_j}$ that $L_{m_j}$ is empty if and only if $a_j$ is even and $m_j$ is odd; if $a_j$ is odd, then $L_{m_j}$ is a one point set for any $m_j$, and if $a_j$ and $m_j$ are even, then $L_{m_j}$ is a two point set.
See Figure~\ref{fig: reflections}. Now, suppose to the contrary that ${\mathbf z} \in \widetilde L_{\mathbf m}$ and $z_j \neq 0$.
Since $z_j \in \Fix \mathcal{R}^{a_j}_{m_j}$, the argument $\arg(z_j)$ is either $m_j\pi/a_j$ or $m_j\pi/a_j + \pi$ (mod $2 \pi$). Therefore the argument $\arg(z_j^{a_j})$ is either $m_j  \pi$ or $(a_j + m_j)\pi$. Since $a_j$ is even and $m_j$ is odd, it follows that $z_j^{a_j} < 0$. This contradicts the definition of $\widetilde L_{\mathbf m}$.
\end{proof}

We now prove Theorem~\ref{thm:top_real}.

\begin{proof}[Proof of Theorem~\ref{thm:top_real}] By Lemma~\ref{lem: deformtononneg}, it suffices to show that
$$
\widetilde L_{\mathbf{m}} \simeq L_{m_0} * L_{m_0} * \dots * L_{m_n}.
$$
We basically follow Step 3 from \cite[Proof of Theorem~1]{Oka73} and take care of  the real structures.

\textbf{Case 1.} \emph{Assume that $L_{m_j}$ is not empty for all $j$.} Define a map $\phi\colon \widetilde L_{\mathbf{m}} \rightarrow L_{m_0} * \cdots * L_{m_n}$ by $\phi({\mathbf z}) = [w_0, t_0, \dots, w_n, t_n]$ where 
$$
w_j = (1/|z_j|)z_j, \quad t_j = z_j^{a_j}.
$$
If $z_j = 0$ for some $j$, then we conventionally interpret $w_j$ to be any point in $L_{m_j}$. Observe that this map is nonetheless not multi-valued due to the equivalence relation in \eqref{eq: joindef}. 
To further check that $\phi$ is well-defined, recall that $\widetilde L_{\mathbf m}$ is given by
\begin{equation}\label{eq: deftildeL}
\widetilde L_{\mathbf m} = \left\{{\mathbf z} \in \C^{n+1} \;|\; \sum_{j=0}^n z_j^{a_j} = 1, \text{ $\mathcal{R}_{m_j}^{a_j}(z_j) = z_j$ and $z_j^{a_j} \geq 0$ for all $j$}\right\}.
\end{equation}
The first and the third conditions in \eqref{eq: deftildeL} show $0 \leq t_j \leq 1$ and $\sum_j t_j = 1$, and we also see that 
$$
\mathcal{R}_{m_j}^{a_j}(w_j) = (1/|z_j|) \mathcal{R}_{m_j}^{a_j}(z_j) = (1/|z_j|)z_j = w_j
$$
by the second equation in \eqref{eq: deftildeL}. Moreover, by the third condition in \eqref{eq: deftildeL}, we have $w_j^{a_j} = 1$ so that $w_j \in L_{m_j}$. Therefore $\phi$ is well-defined. 

Conversely, define a map $\psi\colon L_{m_0} * \cdots * L_{m_n} \rightarrow \widetilde L_{\mathbf{m}}$ by 
$$
\psi([w_0, t_0, \dots, w_n, t_n]) = (t_0^{1/a_0}w_0, \dots, t_n^{1/a_n}w_n).
$$
To check that $\psi$ is well-defined, suppose $(w_0, t_0, \dots, w_n, t_n) \sim (w_0', t_0', \dots, w_n', t_n')$ as in \eqref{eq: joindef}. Then $t_j = t_j'$ for all $j$ and $w_j=w_j'$ for $t_j > 0$. When $t_j=0$, we have $t_j^{1/a_j}w_j = 0 = t_j'^{1/a_j}w_j'$, so $\psi$ does not depend on the choice of representatives. Writing $\psi([w_0, t_0, \dots, w_n, t_n]) = (z_0, \dots, z_n)$ for convenience, we see that $z_0^{a_0} + \cdots + z_n^{a_n} = \sum_j t_j = 1$, and $\mathcal{R}_{m_j}^{a_j}(z_j) = t_j^{1/a_j}\mathcal{R}_{m_j}^{a_j}(w_j) = z_j$. We conclude that $\psi$ is well-defined. 

It is straightforward to see that $\phi$ and $\psi$ are continuous, and a direct computation shows that $\phi \circ \psi = \id_{L_{m_0} * \cdots * L_{m_n}}$ and $\psi \circ \phi = \id_{\widetilde L_{\mathbf m}}$. Therefore $\widetilde L_{\mathbf m}$ is homeomorphic to the join $L_{m_0} * \cdots * L_{m_n}$.

\vspace{0.5em}
\textbf{Case 2.} \emph{Assume that $L_{m_j}$ is empty for some $j$.} Without loss of generality, we may assume that $L_{m_0}, \dots, L_{m_k}$ are empty while the others are not empty for some $0 \leq k \leq n-1$. (It is not possible that $k = n$.) In this case, by Lemma~\ref{lem: emptyforsomej}, if ${\mathbf z} \in \widetilde L_{\mathbf m}$, then $z_j = 0$ for each $0 \leq j \leq k$. In particular $\sum_{j=k+1}^n z_j^{a_j} = 1$.

As in the previous case, we define a map $\phi\colon \widetilde L_{\mathbf m} \rightarrow  L_{m_{k+1}} * \cdots * L_{m_n}$ by $$\phi ({\mathbf z}) = [w_{k+1}, t_{k+1}, \dots, w_n, t_n]$$ where $w_j = (1/|z_j|) z_j$  and $t_j =  z_j^{a_j}$. 
From the above observation, $\sum_{j=k+1}^n t_j = \sum_{j=k+1}^n z_j^{a_j} = 1$, 
and $\phi$ is again a well-defined continuous function. Conversely, define a map $\psi\colon L_{m_{k+1}} * \cdots * L_{m_n} \rightarrow \widetilde L_{\mathbf m}$ by
$$
\psi ([w_{k+1}, t_{k+1}, \dots, w_n, t_n]) = (0, \dots, 0, t_{k+1}^{1/a_{k+1}}w_{k+1}, \dots, t_n^{1/a_n}w_n).
$$ 
Then we have $\phi \circ  \psi = \id_{L_{m_{k+1}} * \cdots * L_{m_n}}$. Indeed, for $[w_{k+1}, t_{k+1}, \dots, w_n, t_n] \in L_{m_{k+1}} * \cdots * L_{m_n}$,
\begin{align*}
(\phi \circ \psi) ([w_{k+1}, t_{k+1}, \dots, w_n, t_n]) &= \phi (0, \dots, 0, t_{k+1}^{1/a_{k+1}} w_{k+1}, \dots, t_n^{1/a_n} w_n) \\
&= [(1/|w_{k+1}|)w_{k+1}, t_{k+1}, \dots, (1/|w_n|)w_n, t_n] \\
&=[w_{k+1}, t_{k+1}, \dots, w_n, t_n]
\end{align*}
where we used the condition $w_j^{a_j} = 1$ which implies $|w_j|=1$.
Conversely, for ${\mathbf z} \in \widetilde L_{\mathbf{m}}$,
\begin{align*}
(\psi \circ \phi) ({\mathbf z}) &= (\psi \circ \phi) (0, \dots, 0, z_{k+1}, \dots, z_n) \\
			& = \psi([(1/|z_{k+1}|)z_{k+1}, z_{k+1}^{a_{k+1}}, \dots, (1/|z_{n}|)z_{n}, z_{n}^{a_{n}}]) \\
			&= (0, \dots, 0, z_{k+1}, \dots, z_n).
\end{align*}
Here, note that $(z_j^{a_j})^{1/a_j} = |z_j|$ for $k+1 \leq j \leq n$. Therefore $\psi \circ \phi = \id_{\widetilde L_{\mathbf m}}$. We conclude that $\widetilde L_{\mathbf m}$ is homeomorphic to the join $L_{m_{k+1}} * \cdots * L_{m_n}$. This completes the proof.
\end{proof}

Recall that the join of two spheres $S^{n} * S^{m}$ is homotopy equivalent to the higher dimensional sphere $S^{n+m+1}$ and that the join of any space with a point is homotopy equivalent to a point. Combining Theorem~\ref{thm:top_real} with Proposition~ \ref{prop: 0dimL}, the homotopy type of our real Lagrangians is determined as follows.

\begin{corollary}
The homotopy type of a real Lagrangian $L_{\mathbf m}$ is either the empty set, a one point set (i.e. contractible), or the $k$-dimensional sphere $S^{k}$ with $0 \leq k \leq n$.
\end{corollary}

\begin{example}
Recall that the Milnor fiber $V(2, \dots, 2)$ is identified with the cotangent bundle $T^*S^n$ as in Example~\ref{ex: invol_quad}.
Consider the real Lagrangian $L_{(0, \dots, 0)}$ which is the fixed point set of the involution $\mathcal{R}_{(0, \dots, 0)}^{(2, \dots, 2)} ({\mathbf z}) = (\overline z_0, \dots, \overline z_n)$.
By Theorem~\ref{thm:top_real}, it is homotopy equivalent to the join $L_0 * \cdots * L_0$. Note that $L_0$ is the fixed point set of the involution $\mathcal{R}_0^2 (z_j) = \overline z_j$ for each $j$ on $V(2)$. Therefore $L_0 \simeq S^0$. We deduce that
$$
L_{(0, \dots, 0)} \simeq L_0 * \cdots * L_0 \simeq S^0 * \cdots * S^0 \simeq S^{n}.
$$ 
This fits with the fact that $L_{(0, \dots, 0)}$ can be identified with the zero section $S^n$ of the cotangent bundle $T^*S^n$.


Similarly, the homotopy type of the real Lagrangian $L_{(0, 1, \dots, 1)}$ can be found using Theorem~\ref{thm:top_real} as
$$
L_{(0, 1, \dots, 1)} \simeq L_0 * L_1 * \cdots * L_1 \simeq S^0 * \emptyset * \cdots * \emptyset \simeq S^0.
$$
Indeed, $L_{(0, 1, \dots, 1)}$ is the disjoint union of two fibers of $T^*S^n$ under the identification $V(2, \dots, 2) = T^*S^n$.
\end{example}





For later use, we remark another useful consequence of Theorem~\ref{thm:top_real}. 
\begin{corollary} \label{cor: real_Lag_app_contrac}
Consider a Milnor fiber $V(a_0, \dots, a_n)$. If $a_j$ is odd for some $j$, then the fixed point set of any involution $\mathcal{R}_{\mathbf m}^{\mathbf a}$ is contractible. If $a_j$ is even for all $j$, then there is an involution $\mathcal{R}_{\mathbf m}^{\mathbf a}$ whose fixed point set is homotopy equivalent to $S^0$. 
\end{corollary}

\begin{proof}
For the first assertion, suppose that $a_0$ is odd without loss of generality. For any involution $\mathcal{R}_{\mathbf m}^{\mathbf a}$,
we know from Theorem~\ref{thm:top_real} that
$$
L_{\mathbf{m}} \simeq L_{m_0} * L_{m_1} * \cdots * L_{m_n}.
$$
Since $a_0$ is odd, $L_{m_0} \subset V(a_0)$ is a one point set. Therefore the join $L_{m_0} * L_{m_1} * \cdots * L_{m_n}$ has the homotopy type of a point. In other words, $L_{\mathbf m}$ is contractible.

Suppose that $a_j$ is even for all $j$. In this case, the real Lagrangian $L_{m_j}$ is either empty or a two point set; more precisely, $L_{m_j} = \emptyset$ if $m_j$ is odd, and $L_{m_j} = S^0$ if $m_j$ is even; see Figure~\ref{fig: reflections}. In particular, by Theorem~\ref{thm:top_real},
$$
L_{(0, 1, \dots, 1)} \simeq L_0 * L_1 * \cdots * L_1 \simeq S^0 * \emptyset * \cdots * \emptyset \simeq S^0
$$
as we wanted.
\end{proof}

\section{Volume growth of fibered twists}

\subsection{Fibered twists} \label{sec: deffiberedtwists} Let $(W, \lda)$ be a Liouville domain such that the contact form $\alpha = \lda|_{\p W}$ on the boundary $\p W$ induces a \emph{periodic} Reeb flow of period $T_P$. We assume $\Ho_c^1(\widehat W; \R) = 0$ so that the compactly supported Hamiltonian diffeomorphism group $\Ham^c(\widehat W)$ coincides with the component of the identity in the compactly supported symplectomorphism group $\Symp^c(\widehat W)$. 

Take a Hamiltonian $H_{\vartheta}\colon \widehat W \rightarrow \R$ satisfying the following properties: For a sufficiently small $\epsilon >0$, $H_{\vartheta} \equiv 0$ on $W \setminus ([1-\epsilon, 1] \times \p W)$, $H_{\vartheta}(r, y)$ depends only on $r$ and is convex in $[1-\epsilon, \infty) \times \p W$, and $H_{\vartheta}'(r) = T_P$ for all $r \in [1, \infty)$. The time-$1$ flow $\vartheta : = \phi_{H_{\vartheta}}^1\colon \widehat W \rightarrow \widehat W$ of the Hamiltonian $H_{\vartheta}$ is called a \emph{fibered twist}. Since the Reeb flow is periodic, $\vartheta$ is  compactly supported. The connected component $[\vartheta] \in \pi_0(\Symp^c(\widehat W))$ does not depend on the choice of the Hamiltonian $H_{\vartheta}$. Geometrically, $\vartheta$ twists $\widehat W$ along the periodic Reeb flow on $\p W$. 

Now, consider a Brieskorn polynomial $f\colon \C^{n+1}  \rightarrow \C$ 
$$
f({\mathbf z}) = z_0^{a_0} + z_1^{a_1} + \cdots + z_n^{a_n}.
$$ 
We define a deformation $W(\mathbf{a})$ of the Milnor fiber $V(\mathbf{a})$ deformed near the origin by
$$
W(\mathbf{a}) = W(a_0, \dots, a_n) : =\{{\mathbf z} \in \C^{n+1} \;|\; f({\mathbf z}) = \rho(|{\mathbf z}|)\} \cap B^{2n+2}
$$
where $\rho\colon \R \rightarrow \R$ is a smooth monotone decreasing function such that $\rho(s) = 1$ near $s = 0$ and $\rho(s) = 0$ near $s =1$. We still call $W(\mathbf{a})$ a \emph{Milnor fiber} of the Brieskorn polynomial. As shown in \cite[Proposition~99]{Fau16}, $W(\mathbf{a})$ is a Liouville domain with the following (weighted) Liouville form
$$
\lda_{\mathbf{a}} : = \frac{1}{2}\sum_{j=0}^{n} a_j(x_j d y_j - y_j dx_j).
$$
The contact boundary $\p W(\mathbf a) = f^{-1}(0) \cap S^{2n+1}$ is called the \emph{link} of the singularity $f$. Denote the induced contact form on $\p W(\mathbf a)$ by $\alpha_{\mathbf a} : = \lda_{\mathbf a}|_{\p W(\mathbf a)}$. A direct computation shows that the corresponding Reeb flow is given by coordinate-wise rotations
\begin{equation}\label{eq: Reebflowformular}
{\mathbf z} \mapsto (e^{it/a_0}z_0, e^{it/a_1}z_1, \dots, e^{it/a_n}z_n).
\end{equation}
In particular, the Reeb flow is periodic with period $T_P = 2 \pi \lcm_{j} a_j$. Therefore, we obtain a fibered twist $\vartheta$ for each Milnor fiber $W(\mathbf a)$.

\begin{remark}
More generally, fibered twists can be defined on Milnor fibers of weighted homogeneous polynomials. See \cite[Section~5]{Se00}.
\end{remark}



\subsection{Growth rate of the real Lagrangians}\label{sec: growthrateofthereallagrangians}  The results on the topology of real Lagrangians in Section~\ref{sec: real_Bries} directly apply to the deformed domain $W(\mathbf{a})$. Indeed, since $\rho$ is a real-valued function, the map $\mathcal{R}_{\mathbf{m}}^{\mathbf{a}}\colon \C^{n+1} \rightarrow \C^{n+1}$ still defines an exact anti-symplectic involution on $W(\mathbf{a})$ for the same reason as in Lemma~\ref{lem: showing_involution}; see also Remark \ref{rem: restrictstotheball}. The proof of \cite[Lemma~4.10]{Se00} shows that there is a diffeomorphism between $W(\mathbf{a})$ and $V(\mathbf{a}) \cap B^{2n+2}$ intertwining the involutions on both sides. Therefore we obtain a correspondence between the real Lagrangians for $\mathcal{R}_{\mathbf{m}}^{\mathbf{a}}$ in $W(\mathbf{a})$ and the ones in $V(\mathbf{a}) \cap B^{2n+2}$ under the diffeomorphism. From now on, we will denote the corresponding real Lagrangians in $W(\mathbf{a})$ by the same letter $L_{\mathbf m}$. In this section, we investigate the growth rate of these Lagrangians $L_{\mathbf m}$ via the Smith inequality in Corollary~\ref{cor: locHW}.

\subsubsection{Stably trivial normal structures for real Lagrangians} \label{sec: stabilytrivialreal} Let $W = W(a_0, \dots, a_n)$ be a Milnor fiber with the Liouville form $\lda_{\mathbf a}$. The map $\sigma\colon W \rightarrow W,\; \sigma({\mathbf z}) = (z_0, \dots, z_{n-1}, -z_n)$ defines an exact symplectic involution on $W$.
Its fixed point set $S = \Fix \sigma$ is given by $\{{\mathbf z} \in W \;|\; z_n = 0\}$, and $S$ can be identified with the lower dimensional Milnor fiber $W(a_0, \dots, a_{n-1})$. Let $\mathcal{R}_{\mathbf{m}}^{\mathbf a} = \mathcal{R}_{(m_0, \dots, m_n)}^{(a_0, \dots, a_n)}\colon W \rightarrow W$ be an anti-symplectic involution, as studied in Section~\ref{sec: realLag}, and $L$ a connected component of the fixed point set of $\mathcal{R}_{\mathbf{m}}^{\mathbf a}$. If $L$ has a boundary, then it is an admissible Lagrangian in $W$. Since $\mathcal{R}_{\mathbf{m}}^{\mathbf a}$ commutes with $\sigma$, the intersection $L_S : = L \cap S$ is the fixed point set of the restriction $\mathcal{R}_{\mathbf{m}}^{\mathbf a}|_{S} = \mathcal{R}_{(m_0, \dots, m_{n-1})}^{(a_0, \dots, a_{n-1})} \colon S \rightarrow S$, which fits into the following diagram.
$$
\begin{tikzcd}
 W(a_0, \dots, a_{n-1}) \arrow[r, hook, "z_n = 0"] \arrow[d, "\mathcal{R}_{\mathbf m}^{\mathbf a}|_{\{z_n = 0\}}"]
    & W(a_0, \dots, a_{n-1}, a_n) \arrow[d, "\mathcal{R}_{\mathbf m}^{\mathbf a}"] \\
  {W(a_0, \dots, a_{n-1})} \arrow[r, hook, "z_n = 0"]
& W(a_0, \dots, a_{n-1}, a_n) \end{tikzcd}
$$

Let $H\colon W \rightarrow W$ be a $\sigma$-invariant admissible Hamiltonian (i.e. the restriction to $W$ of an admissible Hamiltonian on $\widehat W$ which is linear on $\R_{\geq 1} \times \p W$), and denote the time-$1$ Hamiltonian flow by $\phi_H$. The image $\phi_H(L)$ is an admissible Lagrangian in $W$, and the intersection $\phi_H(L) \cap S$, which is an admissible Lagrangian in $S$, is equal to the image $\phi_H(L_S) = \phi_{H_S}(L_S)$ due to the $\sigma$-invariance of $H$.

\begin{theorem}\label{thm: snqmainthm}
In the above setup with $n \geq 3$, if $L_S$ is contractible, then the quadruple  $(W, L, \phi_H(L), \sigma)$ admits a stably trivial normal structure.
\end{theorem}

\begin{proof}
We mimic the argument in \cite[Section~7]{Hen12} which provides a practical way to guarantee the existence of a stably trivial normal structure. 

We divide the argument into four steps. To state the first step, consider the projection $S \times [0, 1] \rightarrow S$, and let $\Upsilon(S)$ be the pullback bundle over $S \times [0,1]$ of the normal bundle $N_W S \rightarrow S$ as in the following diagram.
$$
\begin{tikzcd}
  \Upsilon(S) \arrow[r, dashed] \arrow[d]
    & N_{W}S \arrow[d] \\
  {[0, 1] \times S} \arrow[r]
& S \end{tikzcd}
$$
Comparing with the notation of Section~\ref{sec: def_stabnormal}, note that
\begin{align*}
&TW^{\anti} = \Upsilon(S), \\ 
&L_0 = L, \quad L_1 = \phi_H(L), \\
&L_0^{\inv} = L_S, \quad L_1^{\inv} = \phi_H(L_S).
\end{align*}

\textbf{Step 1}: \emph{The pullback bundle $\Upsilon(S)$ is trivial.} (This step corresponds to \cite[Lemma~7.2]{Hen12}.) Observe that the normal bundle $N_W S$ is a complex line bundle (with respect to a compatible almost complex structure on $W$). The first Chern class $c_1(N_W S)$ satisfies
$$
c_1(TW|_S) = c_1(N_W S) + c_1(TS).
$$
Since $W$ and $S$ are both Milnor fibers, $c_1(TW)$ and $c_1(TS)$ are trivial; see \cite[Corollary~4.13]{Se00}. It follows from the above equation that $c_1(N_W S)$ is likewise trivial. Therefore the line bundle $N_W S$, and hence the pullback bundle $\Upsilon(S)$ is trivial.

\textbf{Step 2}: \emph{The normal bundles $N_{L}L_S$ and $N_{\phi_H(L)} \phi_H(L_S) $ are trivial.} (This step corresponds to \cite[Proposition~7.3]{Hen12}.) Note that $N_L L_S$ is a vector bundle over $L_S$ and that the base space $L_S$ is assumed to be contractible. Therefore $N_L L_S$ is a trivial bundle. Since $\phi_H$ is a diffeomorphism, the image $\phi_H(L_S)$ is also contractible. Therefore the bundle $N_{\phi_H(L)} \phi_H(L_S)$ over it is trivial.

For the next step, let us introduce the subset $X$ of $[0, 1] \times S $ given by
$$
X : = (\{0\} \times L_S) \sqcup (\{1\} \times \phi_H(L_S) ) \subset [0, 1] \times S.
$$

\textbf{Step 3:} \emph{The homomorphism, induced by the inclusion $X \rightarrow [0, 1] \times S$, 
$$
\Ho^i([0, 1] \times S) \rightarrow \Ho^i(X)
$$
is surjective for $i \geq 1$.} As in the previous step, the Lagrangians $L_S$ and $\phi_H(L_S)$ are contractible. Therefore, $X$ has the homotopy type
of two points. It follows that $\Ho^i(X)$ is trivial for $i \geq 1$, and the homomorphism is obviously surjective as asserted.

By borrowing the notation and the argument from  \cite[Proof of Theorem~3.11]{Hen12}, Step 3 shows that the Chern classes of the relative vector bundle $\Upsilon(S)_{\mathrm{rel}}$ are trivial.
Now, we need to show further that the relative vector bundle $[\Upsilon(S)_{\mathrm{rel}}] \in \widetilde K^0(([0, 1] \times S)/X)$ is trivial; this finally implies that $\Upsilon(S)$ carries a stably trivial normal structure. By \cite[Proposition~6.10]{Hen12} (Atiyah--Hirzebruch \cite[Section~2.5]{AtiHir61}), it is now sufficient to prove the following.

\textbf{Step 4}: \emph{For each $i \geq 0$, the cohomology $\Ho^i(([0, 1] \times S)/ X)$ is torsion-free and finitely generated.} Consider the long exact sequence of the pair $([0, 1] \times S,  X)$:
$$
\Ho^{i-1}(X) \rightarrow \Ho^i([0, 1] \times S, X) \rightarrow \Ho^i([0, 1] \times S) \rightarrow \Ho^i(X).
$$
Since $X$ has the homotopy type of two points, it follows that 
$$
\Ho^i(([0, 1] \times S) /X) \cong \Ho^i([0, 1] \times S, X) \cong \Ho^i([0, 1] \times S) \cong \Ho^i(S)
$$
for $i \geq 2$. Since $S$ is a Minor fiber of an isolated complex hypersurface singularity, $S$ is homotopy equivalent to a bouquet of $S^{n-1}$ by \cite[Theorem~6.5]{Mil68}. In particular, $\Ho^i(S)$ is torsion free and finitely generated, and hence so is $\Ho^i(([0, 1] \times S) /X)$ for $i \geq 2$. For the case $i=1$, we consider the following relevant part of the exact sequence
$$
\widetilde{\Ho}^0([0, 1] \times S) \rightarrow \widetilde{\Ho}^0(X) \rightarrow \widetilde{\Ho}^1(([0, 1] \times S)/X) \rightarrow \widetilde{\Ho}^1([0, 1] \times S).
$$
From the facts that $[0, 1] \times S$ is connected, $X$ has the homotopy type of two points, and that $S$ is has the homotopy type of a bouquet of $S^{n-1}$, we deduce
$$
{\Ho}^1(([0, 1] \times S)/X) = \widetilde{\Ho}^1(([0, 1] \times S)/X) \cong \widetilde{\Ho}^0(X) \cong \Z
$$
which is torsion free and finitely generated; here, we use the assumption that $n \geq 3$ to guarantee that $\widetilde{\Ho}^1([0, 1] \times S)$ is trivial. The remaining case $i=0$ is obvious.
\end{proof}


Now we are in position to prove our main application of the Smith inequality in Theorem~\ref{thm: local_fHW}:

\begin{theorem}\label{thm: mainthmgrowthrateHW}
Let $S = W(a_0, \dots, a_{n-1})$ be a Milnor fiber with $n  \geq 3$ and with an involution $\mathcal{R}_{\mathbf{m}}^{\mathbf a}$ on $S$ such that $\Fix \mathcal{R}_{\mathbf{m}}^{\mathbf a}$ is contractible or homotopy equivalent to $S^0$. Let $L_S$ be a connected component of $\Fix \mathcal{R}_{\mathbf{m}}^{\mathbf a}$. Suppose that $\Lambda(L_S) > 0$. Then for any even numbers $b_0, \dots, b_{\ell-1} >0$ with $\ell \geq 1$, the Milnor fiber $W = W(a_0, \dots, a_{n-1}, b_0, \dots, b_{\ell -1})$ admits a contractible admissible Lagrangian $L$ with $\Lambda(L) > 0$.
\end{theorem}

\begin{proof}
We first prove the case when $\ell = 1$ i.e. $W = W(a_0, \dots, a_{n-1}, b_0)$. Note that the involution $\sigma\colon \C^{n+1} \rightarrow \C^{n+1}, \;  \sigma({\mathbf z}, w_0) = (z_0, \dots, z_{n-1}, -w_0)$ restricts to $W$ as $b_0$ is even, and $S = W(a_0, \dots, a_{n-1})$ can be seen as the fixed point set $\Fix \sigma = \{w_0 = 0\} \subset W$. For an integer $k_0$ with $0 \leq k_0 \leq b_0-1$, we consider the involution
$$
\mathcal{R}_{(\mathbf{m}, k_0)}^{(\mathbf{a}, b_0)} = \mathcal{R}_{(m_1, \dots, m_{n-1}, k_0)}^{(a_0, \dots, a_{n-1}, b_0)}\colon W \rightarrow W.
$$
Then the restriction of $\mathcal{R}_{(\mathbf{m}, k_0)}^{(\mathbf{a}, b_0)}$ to $S$ is the given involution $\mathcal{R}_{\mathbf{m}}^{\mathbf a}$, and $L_S$ is a connected component of the real Lagrangian $L_{\mathbf m}$. One can take $k_0$ such that a connected component $L$ of the real Lagrangian $\Fix \mathcal{R}_{(\mathbf{m}, k_0)}^{(\mathbf{a}, b_0)}$ satisfies $L \cap S = L_S$: For example, if $k_0 = 1$, then we have $L_{k_0}^{b_0} = \emptyset$  since $b_0$ is even; see Figure~\ref{fig:reflections_c}. In this case, $L_{(\mathbf{m}, k_0)} \cap S = L_{(\mathbf{m}, k_0)} \cap \{w_0 = 0\} = L_{\mathbf{m}}$. 
Now, since $L_S$ is contractible by assumption, the quadruple $(W, L, \phi_H(L), \sigma)$ admits a stably trivial normal structure by Theorem~\ref{thm: snqmainthm} for any $\sigma$-invariant admissible Hamiltonian $H$. It follows from Corollary~\ref{cor: locHW} that $\Lambda(L) \geq \Lambda(L_S) >0$ as asserted.

To deal with the general case, observe in the previous argument that the integer $k_0$ has been chosen so that the Lagrangian $L$ is contractible; this follows from Theorem~\ref{thm:top_real} with the condition $L_{k_0}^{b_0} = \emptyset$.  Consequently, the Milnor fiber $W = W(a_0, \dots, a_{n-1}, b_0)$ now satisfies all the assumptions for $S$ in the statement where the Lagrangian $L_S$ is replaced by $L$; we have shown that $\Lambda(L) > 0$ in the previous case. Therefore, for an additional even number $b_1$, the Milnor fiber $W(a_0, \dots, a_{n-1}, b_0, b_1)$ still admits a contractible admissible Lagrangian with positive linear growth rate in the wrapped Floer homology. In this way, we can apply the same argument inductively on $\ell \geq 1$, and this yields the general case.
\end{proof}

\begin{example}\label{ex: k+1222}
Consider the $A_k$-type Milnor fiber $W(k+1, 2, 2, 2)$ and the involution $\mathcal{R}_{\mathbf m} = \mathcal{R}_{(0, 1, 1, 1)}$ i.e.
$$
\mathcal{R}_{\mathbf m} ({\mathbf z}) = (\overline z_0, -\overline z_1,- \overline z_2, - \overline z_3).
$$
By Theorem~\ref{thm:top_real}, we can determine the homotopy type of the real Lagrangian $L_{\mathbf m}$:
$$
\displaystyle L_{\mathbf m}  \simeq L_{0} * L_1 *L_1* L_1 \simeq L_0 * \emptyset *  \emptyset * \emptyset \simeq \begin{cases} pt & \text{$k$ even},\\ S^0 & \text{$k$ odd}.    \end{cases}
$$
In fact, the Milnor fiber $W(k+1, 2, 2, 2)$ can be seen as a $k$-linear plumbing of the disk cotangent bundle $DT^*S^3$, and under this identification, $L_{\mathbf m}$ corresponds to one or two cotangent fibers; see \cite[Proposition~3.6]{BaKw21}. 

Let $L$ be a connected component of $L_{\mathbf m}$, which is contractible. As in \cite[Section~7]{KKL18}, we can explicitly compute the filtered wrapped Floer homology $\HW^a(L)$ of $L$ and it has positive linear growth rate $\Lambda(L) >0$. By Theorem~\ref{thm: mainthmgrowthrateHW}, it follows that any higher dimensional extension $W(k+1, 2, 2, 2, b_0, \dots, b_{\ell-1})$, with $b_j$ even for all $j$, admits a contractible admissible Lagrangian whose linear growth rate in wrapped Floer homology is positive; this includes the result in \cite[Corollary~7.2]{KKL18} on higher dimensional $A_k$-type Milnor fibers as a special case.

\end{example}

\subsection{Volume growth of fibered twists} \label{sec: volumegrwothoffiberedtwists} 

Let $(W, \lda)$ be a Liouville domain of dimension $2n$, with $\Ho^1_c(W; \R) = 0$, whose Reeb flow on the boundary is periodic. As in Section~\ref{sec: deffiberedtwists}, we have the connected component $[\vartheta] \in \pi_0(\Symp^c(\widehat W))$ of fibered twists. Now we prove Theorem \ref{thm: intro_volumegrowthviagrowthrate} asserting that the existence of a Lagrangian with positive linear growth rate yields a uniform lower bound for the $n$-dimensional slow volume growth in the component $[\vartheta]$.



\begin{proof}[Proof of Theorem~\ref{thm: intro_volumegrowthviagrowthrate}]
The main ingredient is the Crofton inequality in \cite[Lemma~5.3]{CiGiGu21} for Lagrangian tomographs. We utilize a slight variation formulated in \cite[Section~3]{BaLe22}. For each $m \geq 1$, denote 
$$
L_m : = \varphi^m(L).
$$
Note that $\p L_m = \p L$ for all $m \geq 1$ since $\varphi$ is the identity near the boundary $\p W$. We take a positive isotopy $\varphi_{\text{pos}}\colon W \rightarrow W$, in the sense of \cite[Section~5a]{KhoSei02}, namely $\varphi_{\text{pos}}$ is the time-1 Hamiltonian flow of a Hamiltonian $H$ such that $H$ is zero on the support of $\varphi$, convex near the boundary depending only on the Liouville coordinate $r$, and linear at the end with slope smaller than any Reeb chord on $(\p W, \alpha)$ relative to the Legendrian boundary $\p L$. Then the image
$$
L_m' : = \varphi_{\text{pos}}(L_m)
$$
does not intersect $L$ near the boundary $\p W$ for any $m \geq 1$ while $L_m'$ still coincides with $L_m$ away from the support of $\varphi_{\text{pos}}$. Therefore the completed Lagrangians $\widehat L$ and $\widehat L_m'$ in $\widehat W$ are disjoint on the symplectization part.
By \cite[Lemma~3.2]{BaLe22}, there is a Lagrangian tomograph (see also \cite[Section~5.2]{CiGiGu21}), which in particular means that we have a family of admissible Lagrangians $\{L_s\}_{s \in B}$, where $B$ is a closed ball of sufficiently large dimension, such that
\begin{enumerate}
\item $\widehat L_s$ is Hamiltonian isotopic to $\widehat L$ for each $s \in B$;
\item $\widehat L_s$ intersects  $\widehat L_m'$  transversely for each $m \ge 1$ and almost every $s \in B$ (depending on~$m$).
\end{enumerate}
See Figure~\ref{fig: tomograph} for a schematic description.
\begin{figure*}[h]
\begin{tikzpicture}[scale=1]

\draw[dashed, gray] (0,1.5)--(0,-2);
\node[right] at (0,-0.7) {$r=1$};

\draw [blue, thick] (-4,0)--(4,0);
\node at (-4.5,0) {\color{blue}$\widehat{L}$};
\node at (4.5,0) {\color{red}$\widehat{L}_m$};
\node at (-2,-1.7) {$W$};
\node at (2,-1.7) {$[1,\infty)\times \p W$};
\node[gray] at (-0.5,1) {$\widehat{L}_s$};

\draw [gray] plot [smooth] coordinates {
(-4,1) (-3,1) (-2,1) (-1, 0.8) (-0.2,0)
};
\draw [gray] plot [smooth] coordinates {
(-4,0.8) (-3,0.8) (-2,0.8) (-1, 0.6) (-0.2,0)
};
\draw [gray] plot [smooth] coordinates {
(-4,0.6) (-3,0.6) (-2,0.6) (-1, 0.4) (-0.2,0)
};
\draw [gray] plot [smooth] coordinates {
(-4,0.4) (-3,0.4) (-2,0.4) (-1, 0.2) (-0.2,0)
};
\draw [gray] plot [smooth] coordinates {
(-4,0.2) (-3,0.2) (-2,0.2) (-1, 0) (-0.2,0)
};

\begin{scope}[yscale=-1]
\draw [gray] plot [smooth] coordinates {
(-4,1) (-3,1) (-2,1) (-1, 0.8) (-0.2,0)
};
\draw [gray] plot [smooth] coordinates {
(-4,0.8) (-3,0.8) (-2,0.8) (-1, 0.6) (-0.2,0)
};
\draw [gray] plot [smooth] coordinates {
(-4,0.6) (-3,0.6) (-2,0.6) (-1, 0.4) (-0.2,0)
};
\draw [gray] plot [smooth] coordinates {
(-4,0.4) (-3,0.4) (-2,0.4) (-1, 0.2) (-0.2,0)
};
\draw [gray] plot [smooth] coordinates {
(-4,0.2) (-3,0.2) (-2,0.2) (-1, 0) (-0.2,0)
};	
\end{scope}

\begin{scope}[yshift=0.05cm]
\draw [red,thick] plot [smooth] coordinates {(-4,1) (-3.5,1) (-3,0) (-2.5,-1) (-2.2,0.5) (-1.7, -0.7) (-1.5,0) (-1,0) (-0.5,0) (0,0) (1,0) (4,0)};
\end{scope}

\begin{scope}[yshift=0.05cm]
\draw [thick,red] (-1,0) to[out=0,in=0] (-0.6,0) to[out=45,in=180] (-0.15,0.5) to (0.2, 0.5) to (4,0.5);
\node at (4.5,0.5) {\color{red}$\widehat{L}_m'$};
\end{scope}
\end{tikzpicture}
\caption{The geometric setup of the Lagrangian tomograph}
\label{fig: tomograph}
\end{figure*}
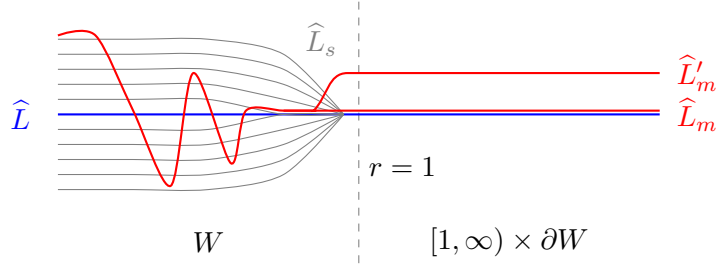
For each $s \in B$ with $\widehat L_s \pitchfork \widehat L_m'$, define
$$
N_m(s) : = |\widehat L_s \cap \widehat L_m'| = |L_s \cap L_m'|.
$$
Equip $B$ with the Lebesgue measure $ds$. The Crofton inequality in \cite[Lemma~3.4]{BaLe22} tells us that there is a constant $C > 0$ such that
\begin{equation}\label{eq: Crof}
\int_{B} N_m(s)ds \leq C \cdot \vol (L_m')
\end{equation}
where the constant $C$ depends on the family $\{L_s\}$, $ds$, and a fixed metric on $\widehat W$ which is compatible with $d\widehat \lda$, but not on $L_m'$. Now, for $k > 0$, we observe that
\begin{align*}
N_m(s) &\geq \dim \LFH(L_s, L_m') \\
             &= \dim \LFH(L, (\varphi_{\text{pos}} \circ \varphi^m)(L))\\
             &= \dim \LFH(L, (\varphi_{\text{pos}} \circ \vartheta^{mk})(L))\\
             &= \dim \HF(H^{mkT_P + \epsilon})\\
             &= \dim \HW^{mkT_P+\epsilon}(L).
\end{align*}
Here, $\LFH$ denotes the Lagrangian (intersection) Floer homology and $H^{mkT_P + \epsilon}$ is an admissible Hamiltonian of slope $mkT_P + \epsilon$ where $T_P$ is the period of the Reeb flow and $\epsilon > 0$ is a sufficiently small number determined by the choice of the positive isotopy $\varphi_{\text{pos}}$. As in \cite[Proof of Theorem~A]{KKL18}, we used the invariance of $\LFH$ under Hamiltonian isotopies and the assumption that $\Ho_c^1(W; \R) = 0$; this implies that $\phi^m(L)$ is Hamiltonian isotopic to $\vartheta^{mk}(L)$. Since $\Lambda(L) > 0$, we have 
\begin{align*}
\liminf_{m \rightarrow \infty} \frac{N_m(s)}{m} &\geq \liminf_{m \rightarrow \infty} \left(\frac{\dim \HW^{mkT_P+\epsilon}(L)}{mkT_P+\epsilon} \cdot \frac{mkT_P+\epsilon}{m} \right) \\ &\geq \liminf_{m \rightarrow \infty} \frac{\dim \HW^{mkT_P+\epsilon}(L)}{mkT_P+\epsilon}  \cdot   \liminf_{m \rightarrow \infty}\frac{mkT_P+\epsilon}{m} \\
&\geq \Lambda(L) \cdot kT_P > 0.
\end{align*}
It follows from the inequality \eqref{eq: Crof} that
$$
\liminf_{m \rightarrow \infty} \frac{\vol(L_m')}{m} > 0. 
$$
Since $L_m'$ is obtained from $L_m$ by applying the positive isotopy $\varphi_{\text{pos}}$ which is supported in a small compact region near the boundary $\p W$, not depending on $m$, this in fact implies that
$$
\liminf_{m \rightarrow \infty} \frac{\vol(L_m)}{m} > 0.
$$
Consequently, we have
$$
s_n(\varphi) \geq \liminf_{m \rightarrow \infty} \frac{\log \vol(L_m)}{\log m} \geq 1
$$
which completes the proof for $k > 0$. 

The case $k< 0$ follows from a straightforward adaptation of the previous argument taking into account the fact that the generating Hamiltonian of the inverse $\vartheta^{-1}$ is now given by $- H_{\vartheta}$; see e.g. \cite[Proof of Theorem~3.4]{CiOa18}. Indeed, using the Poincar\'{e} duality in filtered wrapped Floer homology and cohomology, namely
$$
\HW^{(-\infty, a)}(L) \cong \HW_{(-a, \infty)}(L)
$$
as e.g. in \cite[Theorem~3.4 and Section~8.3]{CiOa18} and \cite[Section~3.8]{Rit}, we obtain the inequality 
$$
N_m(s) \geq \dim \HW^{-(mkT_p + \epsilon)}(L).
$$ 
This yields 
\begin{align*}
\liminf_{m \rightarrow \infty} \frac{N_m(s)}{m} &\geq \liminf_{m \rightarrow \infty} \left(\frac{\dim \HW^{-(mkT_P+\epsilon)}(L)}{-(mkT_P+\epsilon)} \cdot \frac{-(mkT_P+\epsilon)}{m} \right) \\ &\geq \liminf_{m \rightarrow \infty} \frac{\dim \HW^{-(mkT_P+\epsilon)}(L)}{-(mkT_P+\epsilon)}  \cdot   \liminf_{m \rightarrow \infty}\frac{-(mkT_P+\epsilon)}{m} \\
&\geq \Lambda(L) \cdot (-kT_P) > 0.
\end{align*}
The rest of the proof is exactly the same as before.
\end{proof}

Returning to the setup of Theorem~\ref{thm: mainthmgrowthrateHW}, consider a Milnor fiber $S = W(a_0, \dots, a_{n-1})$, $n \geq 3$, with an involution $\mathcal{R}_{\mathbf{m}}^{\mathbf a}$ on $S$ such that $\Fix \mathcal{R}_{\mathbf{m}}^{\mathbf a}$ is contractible or homotopy equivalent to $S^0$. Let $L_S$ be a connected component of $\Fix \mathcal{R}_{\mathbf{m}}^{\mathbf a}$. We suppose that $\Lambda(L_S) > 0$. Let $b_0, \dots, b_{\ell-1} >0$ be even numbers with $\ell \geq 1$ and consider $W = W(a_0, \dots, a_{n-1}, b_0, \dots, b_{\ell-1})$. From Theorems~\ref{thm: intro_volumegrowthviagrowthrate} and~\ref{thm: mainthmgrowthrateHW}, we directly obtain the following consequence.

\begin{corollary}\label{cor: mainapplicationvolgrowth}
In the above situation, for any compactly supported symplectomorphism $\varphi$ with $\varphi \in [\vartheta^k] \in \pi_0(\Symp^c(\widehat W))$ for some $k \neq 0$, we have $s_n(\varphi) \geq  1$.
\end{corollary}

We now give a proof of Theorem~\ref{thm: theorem A}:

\begin{proof}[Proof of Theorem~\ref{thm: theorem A}]
As observed in Example~\ref{ex: k+1222}, the $A_k$-type Milnor fiber $W(k+1, 2, 2, 2)$ fits into the setup of Corollary~\ref{cor: mainapplicationvolgrowth}, regardless of the parity of $k$. Note that any Milnor fiber $W(a_0, \dots, a_n)$ such that
\begin{itemize}
\item at least three $a_j$ are equal to 2, and
\item at most one $a_j$ is odd
\end{itemize}
is of the form $W(k+1, 2, 2, 2, b_0, \dots, b_{\ell-1})$ for $\ell \geq 0$ (no $b_j$ when $\ell = 0$). Therefore Corollary~\ref{cor: mainapplicationvolgrowth} yields the proof.
\end{proof}

\begin{remark}\label{rem: 34222}
Consider the Milnor fiber $V(3, 4, 2, 2, 2, 2)$ which fulfills the assumptions of Theorem~\ref{thm: theorem A} (and is not of $A_k$-type).  A direct computation shows that $\Delta(1) = 1$ where $\Delta$ denotes the characteristic polynomial of the singularity $z_0^3 + z_1^4 + z_2^2 + z_3^2+ z_4^2 + z_5^2$, see e.g. \cite[Theorem~9.1]{Mil68}. This implies that the link of the singularity is a topological sphere by \cite[Theorem~8.5]{Mil68}. As observed in \cite[Section~3.1]{KauKry05}, together with \cite[Section~4c]{Se00}, it follows that the variation operator $\text{Var}(\vartheta)$ of a fibered twist on $V(3, 4, 2, 2, 2, 2)$ is trivial, and $\vartheta$ acts trivially on homology. Therefore in this case the uniform lower bound from Theorem~\ref{thm: theorem A} can be seen as a truly symplectic result.
\end{remark}

\begin{remark}\label{rem: infinite order}
The positivity of the linear growth rate of an admissible Lagrangian implies that the component of fibered twists has infinite order in $\pi_0(\Symp^c(\widehat W))$; see \cite[Theorem B and Remark 5.1]{KKL18}. It therefore follows from Theorem \ref{thm: mainthmgrowthrateHW} and  the above proof that the component of fibered twists in Theorem~\ref{thm: theorem A} has infinite order in $\pi_0  (\Symp^c(V(\mathbf a)))$. A more general result for weighted homogeneous polynomials was shown by Seidel in \cite[Section 4c]{Se00} using different methods.
\end{remark}

\subsection*{Acknowledgement}
The authors cordially thank Felix Schlenk and Otto van Koert for valuable comments, the anonymous referee for kind suggestions.
Joontae Kim was supported by the National Research Foundation of Korea(NRF) grant funded by the Korea government(MSIT) (No. NRF-2022R1F1A1074587) and by the Sogang University Research Grant of 202310013.01.
Myeonggi Kwon was supported by the National Research Foundation of Korea(NRF) grant funded by the Korea government(MSIT) (No. NRF-2021R1F1A1060118).

\bibliographystyle{abbrv}
\bibliography{mybibfile}

\begin{thebibliography}{10}

\bibitem{AboSei}
M.~Abouzaid and P.~Seidel.
\newblock An open string analogue of {V}iterbo functoriality.
\newblock {\em Geom. Topol.}, 14(2):627--718, 2010.

\bibitem{AcuAvd16}
B.~Acu and R.~Avdek.
\newblock Symplectic mapping class group relations generalizing the chain
  relation.
\newblock {\em Internat. J. Math.}, 27(12):1650096, 26, 2016.

\bibitem{AtiHir61}
M.~F. Atiyah and F.~Hirzebruch.
\newblock Vector bundles and homogeneous spaces.
\newblock In {\em Proc. {S}ympos. {P}ure {M}ath., {V}ol. {III}}, pages 7--38.
  American Mathematical Society, Providence, R.I., 1961.

\bibitem{BaKw21}
H.~Bae and M.~Kwon.
\newblock A computation of the ring structure in wrapped {F}loer homology.
\newblock {\em Math. Z.}, 299(1-2):1155--1196, 2021.

\bibitem{BaLe22}
H.~Bae and S.~Lee.
\newblock A comparison of categorical and topological entropies on {W}einstein
  manifolds.
\newblock {\em arXiv preprint}, 2022.

\bibitem{ChiDinvan14}
R.~Chiang, F.~Ding, and O.~van Koert.
\newblock Open books for {B}oothby-{W}ang bundles, fibered {D}ehn twists and
  the mean {E}uler characteristic.
\newblock {\em J. Symplectic Geom.}, 12(2):379--426, 2014.

\bibitem{ChiDinvan16}
R.~Chiang, F.~Ding, and O.~van Koert.
\newblock Non-fillable invariant contact structures on principal circle bundles
  and left-handed twists.
\newblock {\em Internat. J. Math.}, 27(3):1650024, 55, 2016.

\bibitem{CiOa18}
K.~Cieliebak and A.~Oancea.
\newblock Symplectic homology and the {E}ilenberg-{S}teenrod axioms.
\newblock {\em Algebr. Geom. Topol.}, 18(4):1953--2130, 2018.
\newblock Appendix written jointly with Peter Albers.

\bibitem{CiGiGu21}
E.~Cineli, V.~L. Ginzburg, and B.~Z. Gurel.
\newblock Topological entropy of {H}amiltonian diffeomorphisms: a persistence
  homology and {F}loer theory perspective.
\newblock {\em arXiv preprint}, 2022.

\bibitem{Fau16}
A.~Fauck.
\newblock {\em Rabinowitz--Floer homology on Brieskorn manifolds}.
\newblock 2016.
\newblock Thesis (Ph.D.)--Humboldt-Universit\"{a}t zu Berlin.

\bibitem{FraSch05}
U.~Frauenfelder and F.~Schlenk.
\newblock Volume growth in the component of the {D}ehn-{S}eidel twist.
\newblock {\em Geom. Funct. Anal.}, 15(4):809--838, 2005.

\bibitem{Hen12}
K.~Hendricks.
\newblock A rank inequality for the knot {F}loer homology of double branched
  covers.
\newblock {\em Algebr. Geom. Topol.}, 12(4):2127--2178, 2012.

\bibitem{KauKry05}
L.~H. Kauffman and N.~A. Krylov.
\newblock Kernel of the variation operator and periodicity of open books.
\newblock {\em Topology Appl.}, 148(1-3):183--200, 2005.

\bibitem{Kea15}
A.~Keating.
\newblock Lagrangian tori in four-dimensional {M}ilnor fibres.
\newblock {\em Geom. Funct. Anal.}, 25(6):1822--1901, 2015.

\bibitem{KhoSei02}
M.~Khovanov and P.~Seidel.
\newblock Quivers, {F}loer cohomology, and braid group actions.
\newblock {\em J. Amer. Math. Soc.}, 15(1):203--271, 2002.

\bibitem{KKL18}
J.~Kim, M.~Kwon, and J.~Lee.
\newblock Volume growth in the component of fibered twists.
\newblock {\em Commun. Contemp. Math.}, 20(8):1850014, 43, 2018.

\bibitem{KvK}
M.~Kwon and O.~van Koert.
\newblock Brieskorn manifolds in contact topology.
\newblock {\em Bull. Lond. Math. Soc.}, 48(2):173--241, 2016.

\bibitem{Mc18}
M.~McLean.
\newblock Affine varieties, singularities and the growth rate of wrapped
  {F}loer cohomology.
\newblock {\em J. Topol. Anal.}, 10(3):493--530, 2018.

\bibitem{Mil68}
J.~Milnor.
\newblock {\em Singular points of complex hypersurfaces}.
\newblock Annals of Mathematics Studies, No. 61. Princeton University Press,
  Princeton, N.J.; University of Tokyo Press, Tokyo, 1968.

\bibitem{Oba22}
T.~Oba.
\newblock Lefschetz-{B}ott fibrations on line bundles over symplectic
  manifolds.
\newblock {\em Int. Math. Res. Not. IMRN}, (2):1414--1453, 2022.

\bibitem{Oka73}
M.~Oka.
\newblock On the homotopy types of hypersurfaces defined by weighted
  homogeneous polynomials.
\newblock {\em Topology}, 12:19--32, 1973.

\bibitem{Pha65}
F.~Pham.
\newblock Formules de {P}icard-{L}efschetz g\'{e}n\'{e}ralis\'{e}es et
  ramification des int\'{e}grales.
\newblock {\em Bull. Soc. Math. France}, 93:333--367, 1965.

\bibitem{Pol02}
L.~Polterovich.
\newblock Growth of maps, distortion in groups and symplectic geometry.
\newblock {\em Invent. Math.}, 150(3):655--686, 2002.

\bibitem{Rit}
A.~F. Ritter.
\newblock Topological quantum field theory structure on symplectic cohomology.
\newblock {\em J. Topol.}, 6(2):391--489, 2013.

\bibitem{SebTho71}
M.~Sebastiani and R.~Thom.
\newblock Un r\'{e}sultat sur la monodromie.
\newblock {\em Invent. Math.}, 13:90--96, 1971.

\bibitem{Se00}
P.~Seidel.
\newblock Graded {L}agrangian submanifolds.
\newblock {\em Bull. Soc. Math. France}, 128(1):103--149, 2000.

\bibitem{Sei08}
P.~Seidel.
\newblock A biased view of symplectic cohomology.
\newblock In {\em Current developments in mathematics, 2006}, pages 211--253.
  Int. Press, Somerville, MA, 2008.

\bibitem{SS}
P.~Seidel and I.~Smith.
\newblock Localization for involutions in {F}loer cohomology.
\newblock {\em Geom. Funct. Anal.}, 20(6):1464--1501, 2010.

\bibitem{Ul17}
I.~Uljarevic.
\newblock Floer homology of automorphisms of {L}iouville domains.
\newblock {\em J. Symplectic Geom.}, 15(3):861--903, 2017.

\bibitem{Yom87}
Y.~Yomdin.
\newblock Volume growth and entropy.
\newblock {\em Israel J. Math.}, 57(3):285--300, 1987.

\end{thebibliography}

\end{document}